\documentclass[%
  letterpaper,
 onecolumn,
  colorlinks,
]{preprint}

\usepackage[noadjust]{cite}

\usepackage[english]{babel}

\usepackage{amsmath}
\usepackage{amsfonts}
\usepackage{amssymb}
\usepackage{amsthm}
\usepackage{dsfont}
\usepackage{mathtools}

\usepackage{graphicx}

\usepackage{tikz}
\usepackage{pgfplots}
  \pgfplotsset{compat = 1.17}
  \pgfkeys{/pgf/number format/.cd, 1000 sep = {\,}}
  \usepgfplotslibrary{external}
  \tikzset{external/system call = {%
    pdflatex \tikzexternalcheckshellescape
      -halt-on-error
      -interaction=batchmode
      -jobname "\image" "\texsource"}}
  \tikzexternaldisable

\usepackage{subcaption}

\usepackage[linesnumbered, lined, algoruled]{algorithm2e}

\newcommand{\trans}{\ensuremath{\mkern-1.5mu\mathsf{T}}}
\newcommand{\herm}{\ensuremath{\mathsf{H}}}

\DeclareMathOperator{\diag}{diag}
\DeclareMathOperator{\trace}{tr}

\DeclareMathOperator{\real}{Re}
\DeclareMathOperator{\imag}{Im}
\DeclareMathOperator*{\argmin}{arg\,min}
\newcommand*{\defeq}{\ensuremath{\stackrel{\operatorname{def}}{=}}}
\newcommand{\frob}{\ensuremath{\operatorname{F}}}
\DeclareMathOperator{\ds}{d}

\newcommand{\C}{\ensuremath{\mathbb{C}}}
\newcommand{\N}{\ensuremath{\mathbb{N}}}
\newcommand{\Cn}{\ensuremath{\C^{n}}}
\newcommand{\Cp}{\ensuremath{\C^{p}}}
\newcommand{\Cm}{\ensuremath{\C^{m}}}
\newcommand{\Cnn}{\ensuremath{\C^{n \times n}}}

\newcommand{\Cpn}{\ensuremath{\C^{p \times n}}}
\newcommand{\Cpm}{\ensuremath{\C^{p \times m}}}

\newcommand{\R}{\ensuremath{\mathbb{R}}}
\newcommand{\Rn}{\ensuremath{\R^{n}}}
\newcommand{\Rm}{\ensuremath{\R^{m}}}
\newcommand{\Rp}{\ensuremath{\R^{p}}}
\newcommand{\Rnn}{\ensuremath{\R^{n \times n}}}
\newcommand{\Rnm}{\ensuremath{\R^{n \times m}}}
\newcommand{\Rpn}{\ensuremath{\R^{p \times n}}}
\newcommand{\Rnfo}{\ensuremath{\R^{\nfo}}}
\newcommand{\Rnnfo}{\ensuremath{\R^{\nfo \times \nfo}}}
\newcommand{\Rnmfo}{\ensuremath{\R^{\nfo \times m}}}
\newcommand{\Rpnfo}{\ensuremath{\R^{p \times \nfo}}}
\newcommand{\Rnr}{\ensuremath{\R^{n \times r}}}
\newcommand{\Rr}{\ensuremath{\R^{r}}}
\newcommand{\Rrr}{\ensuremath{\R^{r \times r}}}
\newcommand{\Rpr}{\ensuremath{\R^{p \times r}}}
\newcommand{\Rrm}{\ensuremath{\R^{r \times m}}}

\renewcommand{\i}{\ensuremath{\mathfrak{i}}}
\newcommand{\Sys}{\ensuremath{\mathcal{G}}}
\newcommand{\SysFreq}{\Sys^{\operatorname{s}}}
\newcommand{\Sysred}{\ensuremath{\widetilde{\Sys}}}
\newcommand{\Sysfo}{\ensuremath{\Sys_{\fo}}}
\newcommand{\relerr}{\ensuremath{\operatorname{relerr}}}

\newcommand{\nfo}{\ensuremath{n_{\operatorname{fo}}}}
\newcommand{\fo}{\ensuremath{\operatorname{fo}}}
\newcommand{\omegaMin}{\ensuremath{\omega}_{\operatorname{min}}}
\newcommand{\omegaMax}{\ensuremath{\omega}_{\operatorname{max}}}
\newcommand{\leftIso}{\boldsymbol{\Psi}_{\vel}}
\newcommand{\rightIso}{\boldsymbol{\Psi}_{\pos}}

\newcommand{\pos}{\ensuremath{\operatorname{p}}}
\newcommand{\vel}{\ensuremath{\operatorname{v}}}
\newcommand{\inp}{\ensuremath{\operatorname{u}}}
\def\checkmark{\tikz\fill[scale=0.4](0,.35) -- (.25,0) -- (1,.7) -- (.25,.15) -- cycle;}

\newcommand{\BA}{\ensuremath{\boldsymbol{A}}}
\newcommand{\BB}{\ensuremath{\boldsymbol{B}}}
\newcommand{\BC}{\ensuremath{\boldsymbol{C}}}
\newcommand{\BD}{\ensuremath{\boldsymbol{D}}}
\newcommand{\BE}{\ensuremath{\boldsymbol{E}}}
\newcommand{\BK}{\ensuremath{\boldsymbol{K}}}
\newcommand{\BM}{\ensuremath{\boldsymbol{M}}}
\newcommand{\BG}{\ensuremath{\boldsymbol{G}}}

\newcommand{\BJ}{\ensuremath{\boldsymbol{J}}}
\newcommand{\BL}{\ensuremath{\boldsymbol{L}}}
\newcommand{\BP}{\ensuremath{\boldsymbol{P}}}
\newcommand{\BQ}{\ensuremath{\boldsymbol{Q}}}
\newcommand{\BR}{\ensuremath{\boldsymbol{R}}}
\newcommand{\BS}{\ensuremath{\boldsymbol{S}}}
\newcommand{\BT}{\ensuremath{\boldsymbol{T}}}
\newcommand{\BI}{\ensuremath{\boldsymbol{I}}}
\newcommand{\BV}{\ensuremath{\boldsymbol{V}}}
\newcommand{\BU}{\ensuremath{\boldsymbol{U}}}
\newcommand{\BY}{\ensuremath{\boldsymbol{Y}}}
\newcommand{\BSigma}{\ensuremath{\boldsymbol{\Sigma}}}
\newcommand{\BW}{\ensuremath{\boldsymbol{W}}}
\newcommand{\BX}{\ensuremath{\boldsymbol{X}}}

\newcommand{\Be}{\ensuremath{\boldsymbol{e}}}
\newcommand{\Bg}{\ensuremath{\boldsymbol{g}}}

\newcommand{\Bj}{\ensuremath{\boldsymbol{j}}}
\newcommand{\Bk}{\ensuremath{\boldsymbol{k}}}
\newcommand{\Bq}{\ensuremath{\boldsymbol{q}}}

\newcommand{\Bu}{\ensuremath{\boldsymbol{u}}}

\newcommand{\Bw}{\ensuremath{\boldsymbol{w}}}
\newcommand{\Bx}{\ensuremath{\boldsymbol{x}}}
\newcommand{\By}{\ensuremath{\boldsymbol{y}}}
\newcommand{\Bz}{\ensuremath{\boldsymbol{z}}}
\newcommand{\Bzero}{\ensuremath{\boldsymbol{0}}}
\newcommand{\Bone}{\ensuremath{\boldsymbol{1}}}

\newcommand{\tf}{\ensuremath{\BG}}
\newcommand{\BBu}{\ensuremath{\BB_{\inp}}}
\newcommand{\BCp}{\ensuremath{\BC_{\pos}}}
\newcommand{\BCv}{\ensuremath{\BC_{\vel}}}
\newcommand{\BPp}{\ensuremath{\BP_{\pos}}}
\newcommand{\BPv}{\ensuremath{\BP_{\vel}}}
\newcommand{\BQp}{\ensuremath{\BQ_{\pos}}}
\newcommand{\BQv}{\ensuremath{\BQ_{\vel}}}
\newcommand{\BRp}{\ensuremath{\BR_{\pos}}}

\newcommand{\BLv}{\ensuremath{\BL_{\vel}}}

\newcommand{\BRcheckp}{\ensuremath{\skew3\check{\BR}_{\pos}}}
\newcommand{\BLcheckv}{\ensuremath{\skew2\check{\BL}_{\vel}}}
\newcommand{\BUcheck}{\ensuremath{\skew2\check{\BU}}}
\newcommand{\BSigmacheck}{\ensuremath{\check{\BSigma}}}
\newcommand{\BYcheck}{\ensuremath{\skew2\check{\BY}}}

\newcommand{\BMr}{\ensuremath{\skew2\widetilde{\BM}}}
\newcommand{\BDr}{\ensuremath{\skew2\widetilde{\BD}}}
\newcommand{\BKr}{\ensuremath{\skew2\widetilde{\BK}}}

\newcommand{\BCpr}{\ensuremath{\skew2\widetilde{\BC}_{\pos}}}
\newcommand{\BCvr}{\ensuremath{\skew2\widetilde{\BC}_{\vel}}}
\newcommand{\BBur}{\ensuremath{\skew2\widetilde{\BB}_{\inp}}}
\newcommand{\Byr}{\ensuremath{\widetilde{\By}}}

\newcommand{\BGr}{\skew2\ensuremath{\widetilde{\BG}}}

\newcommand{\Bxr}{\skew2\ensuremath{\widetilde{\Bx}}}

\newcommand{\Bvarphi}{\ensuremath{\boldsymbol{\varphi}}}

\newcommand{\Bphi}{\ensuremath{\boldsymbol{\phi}}}

\newcommand{\posTf}{\ensuremath{\BG_{\pos}}}
\newcommand{\velTf}{\ensuremath{\BG_{\vel}}}

\newcommand{\bbL}{\ensuremath{\mathbb{L}}}
\newcommand{\bbLs}{\ensuremath{\mathbb{L}^{\operatorname{s}}}}
\newcommand{\bbLM}{\ensuremath{\mathbb{M}}}

\newcommand{\bbLK}{\ensuremath{\mathbb{K}}}
\newcommand{\bbBu}{\ensuremath{\mathbb{B}_{\inp}}}
\newcommand{\bbCp}{\ensuremath{\mathbb{C}_{\pos}}}
\newcommand{\bbCv}{\ensuremath{\mathbb{C}_{\vel}}}

\newcommand{\CE}{\ensuremath{\mathcal{E}}}

\newcommand{\CC}{\ensuremath{\mathcal{C}}}
\newcommand{\CH}{\ensuremath{\mathcal{H}}}
\newcommand{\CJ}{\ensuremath{\mathcal{J}}}

\newcommand{\objRayleigh}{\ensuremath{\CJ_{k,\operatorname{R}}}}
\newcommand{\BvarphiRayleigh}{\ensuremath{\Bvarphi_{\operatorname{R}}}}
\newcommand{\objStruct}{\ensuremath{\CJ_{k,\operatorname{S}}}}
\newcommand{\BvarphiStruct}{\ensuremath{\Bvarphi_{\operatorname{S}}}}

\newcommand{\pertBeta}{\ensuremath{\Delta_\beta}}

\newcommand{\alphaStar}{\ensuremath{\alpha_\star}}
\newcommand{\betaStar}{\ensuremath{\beta_\star}}
\newcommand{\etaStar}{\ensuremath{\eta_\star}}
\newcommand{\alphaPlus}{\ensuremath{\alpha_+}}
\newcommand{\betaPlus}{\ensuremath{\beta_+}}
\newcommand{\etaPlus}{\ensuremath{\eta_+}}
\newcommand{\alphaMinus}{\ensuremath{\alpha_-}}
\newcommand{\betaMinus}{\ensuremath{\beta_-}}
\newcommand{\etaMinus}{\ensuremath{\eta_-}}
\newcommand{\intAlpha}{\ensuremath{\mathcal{I}_\alpha}}
\newcommand{\intBeta}{\ensuremath{\mathcal{I}_\beta}}
\newcommand{\intEta}{\ensuremath{\mathcal{I}_\eta}}

\newcommand{\soQuadbt}{\ensuremath{\mathsf{soQuadpvBT}}}

\newcommand{\soLoewner}{\ensuremath{\mathsf{soLoewner}}}
\newcommand{\foQuadbt}{\ensuremath{\mathsf{foQuadBT}}}

\newcommand{\Btr}{\ensuremath{\mathsf{BT}}}
\newcommand{\soBt}{\ensuremath{\mathsf{soBT}}}
\newcommand{\sopvBt}{\ensuremath{\mathsf{sopvBT}}}
\newcommand{\SPD}{\ensuremath{SPD}}
\newcommand{\HPD}{\ensuremath{HPD}}
\newcommand{\MSD}{\ensuremath{MSD}}

\theoremstyle{plain}\newtheorem{theorem}{Theorem}
\theoremstyle{plain}
\theoremstyle{definition}\newtheorem{remark}{Remark}

\definecolor{matlabblue}{HTML}{0072BD}
\definecolor{matlaborange}{HTML}{D95319}
\definecolor{matlabyellow}{HTML}{EDB120}
\definecolor{matlabpurple}{HTML}{7E2F8E}
\definecolor{matlabgreen}{HTML}{77AC30}
\definecolor{matlablightblue}{HTML}{4DBEEE}
\definecolor{matlabred}{HTML}{A2142F}

\tikzstyle{sline} = [
  solid,
  line width = 1.5pt
]

\pgfplotscreateplotcyclelist{plotlist}{
  {black, sline},
  {matlabblue, sline, dashed, mark options = {solid}, mark = o, mark repeat = 60, mark phase = 0},
  {matlabpurple, sline, dashdotted, mark options = {solid}, mark = triangle*, mark repeat = 60, mark phase = 10},
  {matlaborange, sline, dashed, mark options = {solid}, mark = *, mark repeat = 60, mark phase = 20},
  {matlabgreen, sline, dotted, mark options = {solid}, mark = square*, mark repeat = 60, mark phase = 30},
  {matlabred, sline, dotted, mark options = {solid}, mark = diamond*, mark repeat = 60, mark phase = 40},
}

\pgfplotscreateplotcyclelist{plateplotlist}{
  {black, sline},
  {matlabblue, sline, dashed, mark options = {solid}, mark = o, mark repeat = 30, mark phase = 0},
  {matlabpurple, sline, dashdotted, mark options = {solid}, mark = triangle*, mark repeat = 30, mark phase = 10},
  {matlaborange, sline, dotted, mark options = {solid}, mark = *, mark repeat = 30, mark phase = 20}
}

\pgfplotscreateplotcyclelist{damping_plotlist}{
  {black, sline},
  {matlabblue, sline, dashed, mark options = {solid}, mark = o, mark repeat = 40, mark phase = 0},
  {matlabpurple, sline, dashdotted, mark options = {solid}, mark = triangle*, mark repeat = 40, mark phase = 10},
  {matlablightblue, sline, dashed, mark options = {solid}, mark = *, mark repeat = 40, mark phase = 20},
  {matlabyellow, sline, dotted, mark options = {solid}, mark = square*, mark repeat = 40, mark phase = 30},
}

\newcommand{\plotfontsize}{\footnotesize}

\begin{document}
  

\title{Data-driven balanced truncation for second-order systems with
  generalized proportional damping}
  
\author[$\ast$]{Sean Reiter}
\affil[$\ast$]{Courant Institute of Mathematical Sciences, New York University
  New York, NY 10012 USA.\authorcr
  \email{s.reiter@nyu.edu}, \orcid{0000-0002-7510-1530}}
  
\author[$\dagger$]{Steffen W. R. Werner}
\affil[$\dagger$]{Department of Mathematics,
  Division of Computational Modeling and Data Analytics, and
  National Security Institute,
  Virginia Tech, Blacksburg, VA 24061, USA.\authorcr
  \email{steffen.werner@vt.edu}, \orcid{0000-0003-1667-4862}}
  
\shorttitle{Data-driven second-order balanced truncation}
\shortauthor{S. Reiter, S. W. R. Werner}
\shortdate{2026-05-22}
\shortinstitute{}
  

  
\abstract{%
  Structured reduced-order modeling is a central component in 
  the computer-aided design of control systems in which cheap-to-evaluate
  low-di\-men\-sion\-al models with physically meaningful internal structures
  are computed.
  In this work, we develop a new approach for the structured data-driven
  surrogate modeling of linear dynamical systems described by second-order time
  derivatives via balanced truncation model-order reduction.
  The proposed method is a data-driven reformulation of position-velocity
  balanced truncation for second-order systems and generalizes the
  quadrature-based balanced truncation for unstructured first-order systems to
  the second-order case.
  The computed surrogates encode a generalized proportional damping structure,
  and we propose a computational procedure for inferring the damping coefficients from data by minimizing a
  least-squares error over the coefficients.
  Several numerical examples demonstrate the effectiveness of the proposed
  method.
}

\novelty{}

\maketitle


\section{Introduction}%
\label{sec:intro}

The mathematical modeling of, e.g., mechanical or electrical
structures~\cite[Ch.~1]{Wer21}, typically results in linear dynamical
systems described by second-order differential equations of the form
\begin{equation} \label{eqn:soSys}
   \Sys: \left\{
    \begin{aligned}
      \BM \ddot{\Bx}(t) + \BD \dot{\Bx}(t) + \BK \Bx(t) & = \BBu \Bu(t),
        \quad \Bx(0) = \dot{\Bx}(0) = \Bzero,\\
      \By(t) & = \BCp \Bx(t) + \BCv \dot{\Bx}(t),
    \end{aligned}
  \right.
\end{equation}
where $\BM, \BD, \BK \in \Rnn$ model the internal system dynamics,
$\BBu \in \Rnm$ describes the external inputs and
$\BCp, \BCv \in \Rpn$ model the quantities of interest.
The three matrices $\BM$, $\BD$ and $\BK$ are typically referred to as
\emph{mass}, \emph{damping} and \emph{stiffness} terms due to their physical
interpretation.
The internal behavior of the system~\cref{eqn:soSys} is given by the states
$\Bx\colon [0, \infty) \to \Rn$, which can be externally influenced by the
inputs $\Bu\colon [0, \infty) \to \Rm$ and the quantities of interest are
observed as outputs $\By\colon [0, \infty) \to \Rp$.
An equivalent description of the input-to-output response of~\cref{eqn:soSys}
is given in the frequency domain via the corresponding transfer function
\begin{equation} \label{eqn:soTf}
  \tf(s) = \left(s\BCv + \BCp\right) \left(s^{2} \BM + s\BD +
    \BK\right)^{-1} \BBu,
\end{equation}
with the complex variable $s \in \C$.
Systems with such second-order internal structure~\cref{eqn:soSys} arise in
a variety of applications, ranging from the vibrational analysis of mechanical
and acoustical structures~\cite{Wer21, AumW23, YueM12} to the modeling of
electro-mechanical systems~\cite{Bil05, Bla18a}.

Due to a generally high demand for modeling accuracy, mathematical models
such as~\cref{eqn:soSys} are typically described by a large number of
differential equations, $n \sim 10^{6}$, making their use in
downstream computational tasks such as numerical simulation or optimization
very challenging.
Simultaneously, in complex applications, explicit state-space models of the
form~\cref{eqn:soSys} may be either difficult to obtain or wholly unavailable. 
Instead, the underlying system can only be accessed in the form of data.
This motivates data-driven reduced-order modeling, which is the
construction of low-dimensional surrogate models solely from system data.
In this work, we consider the problem of learning low-dimensional representations of \emph{structured} second-order
systems as in~\cref{eqn:soSys} directly from input-to-output data.
In other words, our goal is to learn another, so-called reduced-order model of the form
\begin{align} \label{eqn:soSysRed}
  \Sysred: \left\{
    \begin{aligned}
      \BMr \ddot{\Bxr}(t) + \BDr \dot{\Bxr}(t) + \BKr \Bxr(t) & = \BBur \Bu(t),
        \quad \Bxr(0) = \dot{\Bxr}(0) = \Bzero,\\
      \Byr(t) & = \BCpr \Bxr(t) + \BCvr \dot{\Bxr}(t),
    \end{aligned}
  \right.
\end{align}
from evaluations of the transfer function~\cref{eqn:soTf} with the
reduced-order quantities $\BMr, \BDr, \BKr \in \Rrr$, $\BBur \in \Rrm$,
$\BCpr, \BCvr \in \Rpr$, $\Bxr \colon[0, \infty) \to \Rr$, and
$\Byr\colon[0, \infty) \to \Rp$, and having a much smaller number
of differential equations $r \ll n$.

Second-order systems~\cref{eqn:soSys} can always be recast in
unstructured (first-order) form by introducing auxiliary variables.
Thus, it is in principle always possible to apply classical data-driven
modeling techniques to realize unstructured models of the underlying dynamics
in~\cref{eqn:soSys}.
However, while it is always possible to write second-order systems in
first-order form, the converse is in general not true~\cite[Sec.~III]{MeyS96}.
Thus, when using unstructured data-driven modeling methods, the underlying
internal system structure is typically not retrievable.
This is usually an undesired outcome because the internal system structure is
what enables the physical reinterpretation of the learned terms, allows for the
use of established computational tools, and preserves structure-inherent
properties.
Further, twice the number of degrees of freedom is required to realize a
first-order system that recovers the input-to-output response of the 
second-order dynamics, and so structured surrogates typically produce more
accurate approximations compared to unstructured ones having the same model
order.
We refer the reader to~\cite{Wer21, SaaSW19} for comparisons of methods for
structured and non-structured surrogate modeling of~\cref{eqn:soSys}.

In recent years, various methods developed for the intrusive \emph{and}
non-intrusive (data-driven) reduced-order modeling of first-order
linear systems have been extended to the second-order case~\cref{eqn:soSys}.
In the intrusive realm, notable examples include generalizations of the
balanced truncation method~\cite{MeyS96, ChaLVetal06, ReiS08, BenW21b, Wer21}
and transfer function interpolation or moment
matching~\cite{BaiMS05, BaiS05a, BeaG09, Wya12, BeaB14}.
Balancing-based methods are of particular interest due to their advantageous
properties and their interpretation as truncation of internal dynamics
associated with small reachability and observability energies.
On the other hand, for data-driven methods, extensions of the interpolatory
Loewner framework~\cite{PonGB22, SchUBetal18, GosGW24},
rational vector-fitting~\cite{WerGG22},
operator inference~\cite{ShaWK22, ShaNTetal24, FilPGetal23} and
the AAA algorithm~\cite{AckGGetal25} have been developed.
In~\cite{GosGB22}, the authors propose a data-driven reformulation of the
balanced truncation (\Btr) method~\cite{Moo81, MulR76}, called
\emph{quadrature-based balanced truncation} (\foQuadbt{}).
In practice, \foQuadbt{} typically produces (approximate) \Btr{} surrogate
models that are nearly identical to their intrusive counterparts.
However, \foQuadbt{} computes unstructured \emph{first-order} \Btr{} models,
and so a tailored extension suitable for the modeling of second-order systems
is needed.

In this work, we propose an extension of the \foQuadbt{} framework
from~\cite{GosGB22} for the data-driven balanced truncation of
second-order dynamical systems~\cref{eqn:soSys}.
Specifically, our method is a data-driven reformulation of the
\emph{position-velocity balanced truncation} (\sopvBt{}) for
second-order systems~\cref{eqn:soSys} proposed by
Reis and Stykel~\cite{ReiS08}; see \Cref{sec:bg} for details.
The major contributions are described as follows.
\begin{itemize}
  \item In \Cref{sec:soQuadBt}, we introduce the structured data-driven balanced
    truncation framework that we propose, \emph{second-order quadrature-based
    position-velocity balanced truncation} (\soQuadbt).
    The framework builds on \Cref{thm:sobtFromData}, which derives the
    intrusive reduced-order quantities in \sopvBt{} from fre\-quen\-cy-response
    data.
    These data are evaluations of the transfer function~\cref{eqn:soTf} and
    its position- and velocity-output subsystems. 
    The proposed method applies to any system that satisfies a generalized
    proportional damping hypothesis, where $\BD$ is expressed as a linear
    combination of $\BM$ and $\BK$ with potentially frequency-dependent
    coefficients $f$ and $g$.
    
  \item In \Cref{sec:theorConsiderations}, we derive theoretical results
    regarding \soQuadbt{}.
    \Cref{thm:stability} shows that \soQuadbt{} preserves asymptotic stability
    when the data are generated by a second-order system~\cref{eqn:soSys} with
    symmetric positive definite (\SPD{}) mass, damping and stiffness matrices,
    and \Cref{thm:error} derives bounds that control the error in the
    reduced-order model matrices in terms of the associated quadrature error.

  \item Because the damping coefficients $f$ and $g$ are typically not
    known a priori, in \Cref{sec:pracConsiderations}, we propose a
    computational procedure to recover these weights from frequency response
    data in the special cases of Rayleigh and structural damping.
    Specifically, \Cref{thm:grads} provides the gradients needed to recover the underlying
    damping parameters from data.
\end{itemize}
Several numerical examples are presented in \Cref{sec:numerics}, which
validate the proposed data-driven modeling method and demonstrate its
effectiveness.
The work is concluded in \Cref{sec:conclusions}.

Parts of the theoretical results presented in \Cref{sec:soQuadBt} were derived in
the course of writing the dissertation of the corresponding
author~\cite{Rei25} and published therein.
In~\cite{WanYWetal25}, the authors independently derive similar results to those contained in~\cite{Rei25} for the case of second-order \emph{velocity} balanced truncation. 
In contrast to~\cite{WanYWetal25}, our work considers position-velocity balanced truncation and is more general in the sense that it is presented for a
broader class of dynamical systems, namely multiple-input/multiple-output
second-order systems with position as well as velocity outputs and
with the generalized proportional damping model~\cref{eqn:dampingModel}.
Additionally, we note that most of the considerations
in \Cref{sec:theorConsiderations,sec:pracConsiderations} have not
been discussed for the velocity balancing approach in~\cite{WanYWetal25}, namely the stability preservation and the computational procedure for recovering unknown damping parameters.


\section{Mathematical background and preliminaries}%
\label{sec:bg}

In this section, we introduce the general damping model that we consider
in this paper, and we remind the reader of the essential details of balanced
truncation (\Btr{}) model reduction for first and second-order systems. 

\subsection{Generalized proportional damping model}%
\label{sec:damping}

Typically, the mass and stiffness matrices in~\cref{eqn:soSys} are assumed to
be constant, while the modeling of damping, i.e., the dissipation and
conservation of energy in the system, can be significantly more complex.
In particular, mathematical models arising in applications such as
vibro-acoustics that involve structural dynamics, acoustic wave propagation, or
frequency-dependent material properties, only exist in the frequency domain.
These models are described by systems of frequency-dependent
algebraic equations of the form
\begin{equation} \label{eqn:soSysLaplace}
  \SysFreq: \left\{
    \begin{aligned}
      \left(s^2\BM + s\BD(s) + \BK \right)\BX(s) & = \BBu \BU(s),\\
      \BY(s) & =\BCp \BX(s) + \BCv \BX(s).
    \end{aligned}
  \right.
\end{equation}
Here, $\BX\colon\C\to\Cn$, $\BU\colon\C\to\Cm$ and $\BY\colon\C\to\Cp$ are the
Laplace transforms of the state, input, and output vectors
from~\cref{eqn:soSys}, respectively, and the internal damping
$\BD \colon \C \to \Cnn$ is frequency dependent.
See, for example,~\cite[Sec.~2]{AumW23} and~\cite{PasA08} for a more detailed
discussion of the different types of internal damping that may appear in the
modeling of second-order dynamical systems.
We note that in the cases where $\BD$ is constant, the time- and
frequency-domain systems $\Sys$ in~\cref{eqn:soSys} and $\SysFreq$
in~\cref{eqn:soSysLaplace} are in fact equivalent formulations related to
each other via the Laplace transform.

In this work, we allow the damping term to be frequency-dependent and consider
the following generalized form of the proportional damping model
\begin{equation} \label{eqn:dampingModel}
  \BD(s) = f(s) \BM + g(s) \BK.
\end{equation}
Thus, we require $\BD$ to be a linear combination of the mass $\BM$ and
stiffness $\BK$ matrices, where the coefficients $f\colon \C \to \C$ and
$g\colon \C \to \C$ are scalar complex-valued frequency-dependent functions.
This generalized framework includes, in particular, the following two classical
damping models.

\paragraph{Rayleigh damping} Also known as proportional damping, the Rayleigh
damping model is given as a constant linear combination of mass and
stiffness matrices
\begin{equation} \label{eqn:Rayleigh}
  \BD(s) = \alpha \BM + \beta \BK,
\end{equation}
with $\alpha, \beta \geq 0$.
This is covered by the general framework~\cref{eqn:dampingModel} by setting
$f(s) = \alpha \geq 0$ and $g(s) = \beta \geq 0$.
The damping parameters $\alpha$ and $\beta$ allow one to choose the frequency range
in which damping is applied, as well as the intensity of the damping.

\paragraph{Structural damping} Also known as hysteretic damping, the structural
damping model describes a constant damping effect over the full frequency range
using
\begin{equation} \label{eqn:structural}
  \BD(s) = \i \frac{\eta}{s} \BK,
\end{equation}
where $\eta \geq 0$ is a material-dependent structural loss factor.
In the general framework~\cref{eqn:dampingModel}, this damping model is given
by setting $f(s) = 0$ and $g(s) = \i \tfrac{\eta}{s}$.


\subsection{Balanced truncation model reduction}%
\label{sec:foBt}

Consider first-order linear systems of the form
\begin{equation} \label{eqn:foSys}
  \Sysfo: \left\{\begin{aligned}
     \BE \dot{\Bq}(t) & = \BA \Bq(t) + \BB \Bu(t),
     \quad \Bq(0) = \Bzero, \\
     \By(t) & = \BC \Bq(t),
  \end{aligned}\right.
\end{equation}
where $\BE, \BA \in \Rnnfo$, $\BB \in \Rnmfo$, $\BC \in \Rpnfo$ 
and $\Bq \colon [0, \infty) \to \Rnfo$.
We assume that $\BE$ is nonsingular, that the realization of the
system~\cref{eqn:foSys} is controllable and observable, and that the
system~\cref{eqn:foSys} is asymptotically stable, i.e., that the eigenvalues of
the matrix pencil $\lambda \BE - \BA$ lie in the open left half of the complex
plane.
The core idea of \Btr{} is the simultaneous diagonalization of two symmetric
positive semi-definite matrices, called system Gramians, and the truncation of
states corresponding to small eigenvalues of these
Gramians~\cite{MulR76, Moo81}.
The system Gramians $\BP, \BE^{\trans} \BQ \BE \in \Rnnfo$, also called the
\emph{controllability} and \emph{observability Gramians}, respectively, are
defined via
\begin{subequations} \label{eqn:foGramians}
\begin{align} \label{eqn:contGramian}
  \BP & =\frac{1}{2\pi} \int\limits_{-\infty}^{\infty} (\i z \BE -\BA)^{-1} \BB
    \BB^{\trans}(\i z \BE - \BA)^{-\herm} \ds z, \quad \text{and}\\
  \label{eqn:obsvGramian}
  \BQ & = \frac{1}{2\pi} \int\limits_{-\infty}^{\infty} (\i z \BE - \BA)^{-\herm}
    \BC^{\trans} \BC (\i z \BE - \BA)^{-1} \ds z.
\end{align}
\end{subequations}
Under the assumptions stated above, the matrices $\BP$ and $\BQ$ are \SPD{}, and thus admit Cholesky factorizations of the form
$\BP = \BR \BR^{\trans}$ and $\BQ = \BL \BL^{\trans}$ where $\BR,\BL\in\Rnnfo$.
The system-theoretic significance of these matrices is that $\BP$ and
$\BE^{\trans} \BQ \BE$ respectively quantify the controllability and
observability of a state $\Bq$.
Then, \Btr{} computes a reduced-order model by removing components of the state
space that correspond to small controllability and observability energies in
the balanced basis.
The appeal of \Btr{} comes from the preservation of asymptotic
stability in the reduced-order model~\cite{PerS82} and the associated
a priori error bound on the $\CH_{\infty}$ norm~\cite{Enn84}.
For further details on balancing-related model reduction for first-order linear
systems~\cref{eqn:foSys}, we refer the reader to~\cite{Ant05, GugA04, BenB17}
and references therein.


\subsection{Balanced truncation for second-order systems}
\label{sec:soBt}

There have been several attempts to generalize the ideas of classical \Btr{} to
second-order systems~\cite{MeyS96, ReiS08, ChaLVetal06, Wer21}. 
The unifying feature of the proposed methods is the reformulation of the
second-order system in~\cref{eqn:soSys} as an equivalent $2n$-dimensional
first-order system~\cref{eqn:foSys}.
This can always be accomplished by introducing the first-order state vector
$\Bq^{\trans} \defeq \begin{bmatrix} \Bx^{\trans} &
\dot{\Bx}^{\trans}\end{bmatrix}$ and re-organizing the second-order dynamics
in~\cref{eqn:soSys} accordingly.
The resulting $2n$-dimensional system~\cref{eqn:foSys} is given in the
so-called \emph{first companion form}, with the system matrices
\begin{equation} \label{eqn:foRealization}
  \BE = \begin{bmatrix} \BI_{n} & \Bzero \\ \Bzero & \BM \end{bmatrix}, \quad
  \BA = \begin{bmatrix} \phantom{-}\Bzero & \phantom{-}\BI_{n} \\ -\BK & -\BD
    \end{bmatrix}, \quad
  \BC = \begin{bmatrix}  \BCp & \BCv \end{bmatrix}, \quad
  \BB = \begin{bmatrix} \Bzero \\ \BBu \end{bmatrix},
\end{equation}
where $\BI_n$ denotes the $n\times n$ identity matrix.
As in \Cref{sec:foBt}, we assume throughout this paper that the mass matrix
$\BM$ in~\cref{eqn:soSys} is nonsingular, that the realization of the
system~\cref{eqn:soSys} is controllable and observable, and that the
system~\cref{eqn:soSys} is asymptotically stable, i.e., the eigenvalues of the
matrix pencil $\lambda^{2} \BM + \lambda \BD + \BK$ lie in the open left half
of the complex plane.
Following these assumptions about the second-order system~\cref{eqn:soSys}, the corresponding first-order realization~\cref{eqn:foRealization}
has a nonsingular $\BE$ matrix, and is asymptotically stable as well as
controllable and observable~\cite{LauA84}.

Let the first-order system Gramians in~\cref{eqn:foGramians}
be partitioned according to the block structure in the companion
form~\cref{eqn:foRealization} as
\begin{equation} \label{eqn:gramPartition}
  \BP=\begin{bmatrix} \BPp & \BP_{12}\\[2pt] \BP_{12}^{\trans} & \BPv
    \end{bmatrix}
  \quad\text{and}\quad
  \BE^{\trans} \BQ \BE = \begin{bmatrix} \BQp & \BQ_{12}\BM\\[2pt]
    \BM^{\trans}\BQ_{12}^{\trans} & \BM^{\trans} \BQv \BM
    \end{bmatrix}.
\end{equation}
The submatrices $\BPp, \BQp \in \Rnn$ are the so-called
\emph{position-controllability and position-ob\-serv\-abil\-i\-ty Gramians},
while $\BPv, \BM^{\trans} \BQv \BM \in \Rnn$ are the
\emph{velocity-controllability and velocity-observability Gramians}.
Because $\BP$ and $\BQ$ are \SPD{}, so too are their diagonal submatrices.
The Gramians $\BPp$ and $\BM^{\trans} \BQv \BM$ can alternatively be expressed
via contour integrals in terms of the input-to-state and state-to-output
mappings of the system~\cref{eqn:soSys} in the frequency domain
\begin{subequations} \label{eqn:soGramians}
\begin{align} \label{eqn:PpIntegral}
  \BPp & = \frac{1}{2\pi} \int_{-\infty}^{\infty}
    \left(-z^{2} \BM + \i z \BD + \BK \right)^{-1} \BBu \BBu^{\trans}
    \left(-z^{2} \BM + \i z \BD + \BK \right)^{-\herm} \ds z,\\
  \label{eqn:QvIntegral}
  \BQv & = \frac{1}{2\pi} \int_{-\infty}^{\infty}
    \left(-z^{2} \BM + \i z \BD + \BK \right)^{-\herm}
      (\BCp + \i z \BCv)^{\herm} \\
  & \nonumber \hphantom{\frac{1}{2\pi} \int_{-\infty}^{\infty}}
    \qquad{}\times{}
      (\BCp + \i z \BCv)
    \left(-z^{2} \BM + \i z \BD + \BK \right)^{-1} \ds z;
\end{align}
\end{subequations}
see, e.g.,~\cite[Prop.~2.1]{PrzPB24} and~\cite[Sec.~4.4]{Bre16}.
The methods proposed in~\cite{MeyS96, ReiS08} correspond to balancing selected
combinations of the second-order controllability and observability Gramians
in~\cref{eqn:gramPartition} with subsequent truncation.
In particular, the so-called \emph{free-velocity balanced truncation}
from~\cite{MeyS96} and the \emph{position-velocity balanced truncation}
(\sopvBt{}) from~\cite{ReiS08} correspond to the simultaneous diagonalization
of $\BPp, \BQp$ and $\BPp, \BM^{\trans} \BQv \BM$, respectively.
Algorithmically, balancing and truncation are achieved simultaneously by using
an appropriate generalization of the square-root algorithm from the first-order
system case~\cite{LauHPetal87, TomP87}.
For a more detailed overview of balanced truncation for second-order systems
and all the different variations, we refer the reader
to~\cite{Wer21, SaaSW19, BenW21b}.

The primary focus of this work lies on the \sopvBt{} method from~\cite{ReiS08}.
The main steps of the approach are summarized in \Cref{alg:pvBT}.
In the case of \emph{symmetric} second-order systems, i.e., the mass, damping,
and stiffness matrices in~\cref{eqn:soSys} are \SPD{} with $\BBu = \BCp^{\trans}$ and $\BCv = \Bzero$,
asymptotic stability is preserved by \sopvBt{}.
In fact, in this particular situation, we have that $\BPp = \BQv$;
see~\cite[Thm.~3.1]{ReiS08}.
Note also that \Cref{alg:pvBT} can be applied if either the position
or velocity-output terms in~\cref{eqn:soSys} are zero.

\begin{algorithm}[t!]
  \SetAlgoHangIndent{1pt}
  \DontPrintSemicolon
  \caption{Second-order position-velocity \Btr{} (\sopvBt{}).} 
  \label{alg:pvBT}

  \KwIn{Second-order system matrices $\BM, \BD, \BK, \BBu, \BCp, \BCv$,
    reduction order $r$.}
  \KwOut{Reduced-order system matrices $\BMr, \BDr, \BKr, \BBur, \BCpr, \BCvr$
    in~\cref{eqn:soSysRed}.}
  
  Compute Cholesky factorizations $\BP = \BR \BR^{\trans}$ and
    $\BQ = \BL \BL^{\trans}$ of the system
    Gramians~\cref{eqn:foGramians} using the companion
    form~\cref{eqn:foRealization}, where the factors are partitioned as
    \begin{equation*}
      \BR = \begin{bmatrix} \BRp \\ \star \end{bmatrix}
      \quad\text{and}\quad
      \BL = \begin{bmatrix} \star \\ \BLv \end{bmatrix},
    \end{equation*}
    with $\BRp, \BLv \in \Rnn$; the $\star$ marks submatrices that are not
    needed.\;
    
  Compute the singular value decomposition 
    \begin{equation*}
      \BLv^{\trans} \BM \BRp =
        \begin{bmatrix} \BU_{1} & \BU_{2} \end{bmatrix}
        \begin{bmatrix} \BS_1 & \Bzero \\ \Bzero & \BS_{2} \end{bmatrix}
        \begin{bmatrix} \BV_{1}^{\trans} \\[2pt] \BV_{2}^{\trans} \end{bmatrix},
    \end{equation*}
    where $\BU_{1}, \BV_{1} \in \Rnr$ have orthonormal columns and
    $\BS_{1} \in \Rrr$ is diagonal containing the $r$ largest singular values.\;
    
  Using $\BW = \BLv \BU_{1} \BS_{1}^{-1/2}$ and $\BT = \BRp \BV_{1} \BS_{1}^{-1/2}$ compute the reduced-order matrices by projection 
    \begin{subequations} \label{eqn:soBTProj}
    \begin{alignat}{2}
      \BMr & = \BW^{\trans} \BM \BT && = \BI_{r}, \\
      \BDr & = \BW^{\trans}\BD\BT && = \BS_{1}^{-1/2} \BU_{1}^{\trans}
        \left( \BLv^{\trans} \BD \BRp \right) \BV_{1} \BS_{1}^{-1/2}, \\
      \BKr & = \BW^{\trans} \BK \BT && = \BS_{1}^{-1/2} \BU_{1}^{\trans}
        \left( \BLv^{\trans} \BK \BRp \right) \BV_{1} \BS_{1}^{-1/2}, \\
      \BBur & = \BW^{\trans} \BBu && = \BS_{1}^{-1/2} \BU_{1}^{\trans}
        \left( \BLv^{\trans} \BBu \right), \\
      \BCpr & = \BCp \BT && = \left( \BCp \BRp \right)
        \BV_{1} \BS_{1}^{-1/2}, \\
      \BCvr & = \BCv \BT && = \left( \BCv \BRp \right) \BV_{1} \BS_{1}^{-1/2}.
    \end{alignat}
    \end{subequations}\;
    \vspace{-\baselineskip}
\end{algorithm}

\begin{remark} \label{remark:structureIntegrals}
  In the more complex case of frequency-dependent
  damping~\cref{eqn:soSysLaplace}, one can still formulate structured Gramians
  via the integral representations in~\cref{eqn:soGramians}
  by replacing the constant $\BD$ with the frequency-dependent $\BD(s)$.
  Then, if the integrals converge, the structured Gramians can be approximated
  by numerical quadrature rules, and \Cref{alg:pvBT} can be applied;
  see, for example, the work on structure-preserving model reduction of
  integral-differential equations in~\cite{Bre16}.
\end{remark}

\begin{remark} \label{remark:whyPv}
  The reason only \sopvBt{} is considered here is due to the integral
  representations of the second-order system
  Gramians~\cref{eqn:soGramians}.
  These representations are given in terms of the input-to-state and
  state-to-output maps of the second-order system~\cref{eqn:soSys}, and they
  will allow us to facilitate a data-driven reformulation of the method expressed
  in terms of transfer function data~\cref{eqn:soTf}.
  Even though one can also write down integral formulations of $\BPv$ and
  $\BQp$, the structure of these integrals is not as amenable to the tools
  used in \Cref{sec:soQuadBt}.
  In particular, it is not clear what the ``data'' needed in the computations
  correspond to.
\end{remark}


\section{Data-driven second-order balanced truncation}%
\label{sec:soQuadBt}

The \sopvBt{} method in \Cref{alg:pvBT} is intrusive in the sense that it
requires an explicit state-space model of the second-order
system~\cref{eqn:soSys} to compute the Cholesky factors $\BRp, \BLv \in \Rnn$
and to compute the reduced-order model~\cref{eqn:soSysRed} by
projecting the full-order model matrices~\cref{eqn:soBTProj}.
In this section, we present a data-driven reformulation of \sopvBt{}. 
Our only assumption is that we can sample the full-order transfer
function~\cref{eqn:soTf} in a prescribed range of frequencies along the
imaginary axis $\i\R$ as well as the coefficient functions
in~\cref{eqn:dampingModel}, which we assume for now are known a priori.
The resulting method, which we call \emph{second-order quadrature-based
position-velocity balanced truncation} (\soQuadbt{}), is a generalization of
the \foQuadbt{} framework for first-order systems from~\cite{GosGB22}.
Our method allows for the computation of (approximate) \sopvBt{}-based second-order
surrogate models~\cref{eqn:soSysRed} directly from frequency-response data.

We begin with the following observation:
The matrices $\BU_{1}$, $\BY_{1}$ and $\BSigma_{1}$ in~\cref{eqn:soBTProj} stem from the
truncated singular value decomposition of $\BLv^{\trans} \BM \BRp$.
Under the assumption that $\BD(s) = f(s)\BM + g(s)\BK$, it follows that the
intrusive \sopvBt{} reduced-order model constructed via \Cref{alg:pvBT} is
fully characterized by the five quantities
\begin{equation} \label{eqn:soBTKeyQuantities}
  \BLv^{\trans} \BM \BRp, \quad
  \BLv^{\trans} \BK \BRp, \quad
  \BLv^{\trans} \BBu, \quad
  \BCp\BRp \quad\text{and}\quad
  \BCv\BRp.
\end{equation}
Along with the integral formulations of the position-reachability and
velocity-ob\-serv\-abil\-i\-ty Gramians $\BPp\in\Rnn$ and
$\BM^{\trans}\BQv\BM\in\Rnn$ in~\cref{eqn:soGramians}, 
this observation suggests a natural extension of \foQuadbt{} for the
second-order system~\cref{eqn:soSys}.
In~\cite{GosGB22}, the authors arrive at a fully data-driven reformulation of
the first-order \Btr{} method by replacing the exact
square-root factors of $\BR$ and $\BL$ with low-rank factors derived from
numerical quadrature rules applied to the integral representations of the
Gramians~\cref{eqn:foGramians}.
The quadrature-based factors are never explicitly formed, and the method only
makes reference to transfer function evaluations at the underlying
quadrature nodes.
We use this idea as a blueprint to develop a data-driven reformulation
of \sopvBt{}.
Specifically, we use implicit numerical quadrature rules to derive
low-rank approximations of the exact factors $\BRp$ and $\BLv$, and ultimately
show how to realize the quadrature-based approximations of the intrusive
reduced-order quantities in~\cref{eqn:soBTKeyQuantities} non-intrusively using
only transfer function data.

Recall the integral representations of the
Gramians~\cref{eqn:soGramians}, but with the
frequency-dependent $\BD(s)$ in place of $\BD$. 
For compactness of the presentation, we introduce the notation
\begin{equation} \label{eqn:soPencil}
  \Bvarphi \colon \C \to \Cnn \quad\text{with}\quad
  \Bvarphi(s) \defeq s^{2} \BM + s \BD(s) + \BK
\end{equation}
for the resolvent term of the second-order transfer function~\cref{eqn:soTf}.
Consider a numerical quadrature rule defined by the nodes
$\i \zeta_{1}, \ldots, \i \zeta_{J} \in \i\R$ and the corresponding weights
$\varrho_{1}^{2}, \ldots, \varrho_{J}^{2} \in \R$.
Applying this quadrature rule to $\BPp$ in~\cref{eqn:PpIntegral} reveals the
approximate factorization
\begin{equation*}
  \BPp \approx \sum\limits_{j = 1}^{J} \varrho_{j}^{2}
    \Bvarphi(\i \zeta_{j})^{-1} \BBu
    \left( \Bvarphi(\i \zeta_{j})^{-1} \BBu \right)^{\herm}
  = \BRcheckp \BRcheckp^{\herm},
\end{equation*}
where the factor $\BRcheckp \in \C^{n \times mJ}$ is defined as
\begin{equation} \label{eqn:soQuadContFactor}
  \BRcheckp \defeq
    \begin{bmatrix} \varrho_{1} \Bvarphi(\i \zeta_{1})^{-1} \BBu &
    \varrho_{2} \Bvarphi(\i \zeta_{2})^{-1} \BBu & \ldots &
    \varrho_{J} \Bvarphi(\i \zeta_{J})^{-1} \BBu \end{bmatrix}.
\end{equation}
A similar factorization can be derived for $\BQv$ in~\cref{eqn:QvIntegral}.
Given another numerical quadrature rule defined by the nodes
$\i \vartheta_{1}, \ldots, \i \vartheta_{K} \in \i \R$ and the corresponding
weights $\varphi_{1}^{2}, \ldots, \varphi_{K}^{2} \in \R$, applying this quadrature rule then yields
\begin{equation*}
  \BQv\approx \sum\limits_{k = 1}^{K} \varphi_{k}^{2}
    \left( \left(\BCp + \i \vartheta_{k} \BCv \right)
    \Bvarphi(\i \vartheta_{k})^{-1} \right)^{\herm}
    \left( \BCp + \i \vartheta_{k} \BCv \right)
    \Bvarphi(\i \vartheta_{k})^{-1} = \BLcheckv \BLcheckv^{\herm},
\end{equation*}
where the factor $\BLcheckv\in\C^{n \times p K}$ is defined according to
\begin{equation} \label{eqn:soQuadObsvFactor}
  \BLcheckv^{\herm} \defeq
    \begin{bmatrix} \varphi_{1} \left( \BCp +\i \vartheta_{1} \BCv \right)
      \Bvarphi(\i \vartheta_{1})^{-1} \\
      \varphi_{2} \left( \BCp + \i \vartheta_{2} \BCv \right)
      \Bvarphi(\i \vartheta_{2})^{-1}\\[2pt]
      \vdots\\[2pt]
      \varphi_{K} \left( \BCp + \i \vartheta_{K} \BCv \right)
      \Bvarphi(\i \vartheta_{K})^{-1}
    \end{bmatrix}.
\end{equation}
Replacing the exact Cholesky factors $\BRp$ and $\BLv$ in \Cref{alg:pvBT} with
the quadrature-based factors $\BRcheckp$ and $\BLcheckv$
from~\cref{eqn:soQuadContFactor,eqn:soQuadObsvFactor} in forming the
matrices~\cref{eqn:soBTKeyQuantities} already yields a low-rank
implementation of \sopvBt.
As we show next in \Cref{thm:sobtFromData}, the significant implication of this
low-rank formulation is that the quadrature-based approximations to the
intrusive reduced-order quantities in~\cref{eqn:soBTKeyQuantities} can be
computed non-intrusively from transfer function samples and evaluations of
the damping coefficients $f$ and $g$ in~\cref{eqn:dampingModel}.

Before moving on to the results, we introduce the following
position and velocity-output transfer functions
$\posTf \colon \C \to \Cpm$ and $\velTf \colon \C \to \Cpm$ of the
system~\cref{eqn:soSys}, defined as
\begin{equation} \label{eqn:pvTfs}
  \posTf(s) \defeq \BCp \Bvarphi(s)^{-1} \BBu
  \quad\text{and}\quad
  \velTf(s) = s \BCv \Bvarphi(s)^{-1} \BBu.
\end{equation}
The two functions $\posTf$ and $\velTf$ are each the transfer function of a
second-order system~\cref{eqn:soSys} with purely position and velocity outputs,
respectively.
In general, we have that $\BG(s) = \posTf(s) + s \velTf(s)$.
For ease of presentation, we introduce the following notation for indexing the
entries of block matrices:
Given a block matrix $\BX \in \C^{p K \times m J}$, we use  
\begin{equation*}
  \BX_{\Bk, \Bj} \defeq \BX_{(k - 1)p + 1 : k p, (j - 1)m + 1 : j m}
    \in \C^{K \times J},
\end{equation*}
to denote the block submatrix of $\BX$ containing the rows
$(k - 1)p + 1, \ldots, k p$ and the columns $(j - 1)m + 1, \ldots, j m$.
If $p = 1$ (or $m = 1$), we instead write $\BX_{:, \Bj}$ (or $\BX_{\Bk, :}$)
to denote the $j$-th (or $k$-th) block entry of $\BX$.

To simplify the proof of the main result \Cref{thm:sobtFromData}, we recall the well-known resolvent identities
\begin{subequations}
\begin{align} \label{eqn:resolventId1}
(s \BX + \BY)^{-1} \BX (z \BX + \BY)^{-1} & =
  \frac{(z \BX + \BY)^{-1} - (s \BX + \BY)^{-1}}{s - z},\\
\label{eqn:resolventId2}
(s \BX + \BY)^{-1} \BY (z \BX + \BY)^{-1} & =
  \frac{z (z \BX + \BY)^{-1} - s (s \BX + \BY)^{-1}}{z - s}
\end{align}
\end{subequations}
which hold for any matrices $\BX,\BY\in\Cnn$ and $s, z\in\C$ with $s \neq z$ and such that $s\BX + \BY$ and
  $z\BX + \BY$ are nonsingular; see, e.g.,~\cite{DunS88}.

\begin{theorem} \label{thm:sobtFromData} \allowdisplaybreaks
  Define the functions $d \colon \C \to \C$, $n \colon \C \to \C$ and
  $h \colon \C \to \C$ as
  \begin{equation} \label{eqn:auxFunctions}
    d(s) \defeq 1 + s g(s), \quad
    n(s) \defeq s^{2} + s f(s) \quad\text{and}\quad
    h(s) \defeq \frac{n(s)}{d(s)},
  \end{equation}
  where $f$ and $g$ are the damping coefficient functions
  from~\cref{eqn:dampingModel}.
  Suppose the left and right quadrature nodes
  $\i \vartheta_{1}, \ldots, \i \vartheta_{K}$ and
  $\i \zeta_{1}, \ldots, \i \zeta_{J}$
  in~\cref{eqn:soQuadContFactor,eqn:soQuadObsvFactor} are such that
  \begin{equation*}
    \zeta_{j} \neq 0, \quad
    d(\i \vartheta_{k}) \neq 0, \quad
    d(\i \zeta_{j})\neq 0, \quad
    h(\i \vartheta_{k}) \neq h(\i \zeta_{j}), 
  \end{equation*}
  and so that the matrices $\Bvarphi(\i \zeta_{j})$ and
  $\Bvarphi(\i \vartheta_{k})$ are invertible for all
  $j = 1, \ldots, J$ and $k = 1,\ldots, K$, where $\Bvarphi$ is
  defined in~\cref{eqn:soPencil}.
  Let the quadrature-based factors $\BRcheckp$ and $\BLcheckv$ be given as
  in~\cref{eqn:soQuadContFactor,eqn:soQuadObsvFactor}, and define the matrices
  \begin{subequations} \label{eqn:soDataMatrices}
  \begin{alignat}{2}
    \bbLM & \defeq \BLcheckv^{\herm} \BM \BRcheckp \in \C^{pK \times mJ},
      \qquad &
    \bbLK & \defeq \BLcheckv^{\herm} \BK \BRcheckp \in \C^{pK \times mJ},\\
    \bbBu & \defeq \BLcheckv^{\herm} \BBu \in \C^{pK \times m}, 
      \qquad &
    \bbCp & \defeq \BCp \BRcheckp \in \C^{p \times mJ}, \\
    \bbCv & \defeq \BCv \BRcheckp \in \C^{p \times mJ}.
  \end{alignat}
  \end{subequations}
  Then, the matrices~\cref{eqn:soDataMatrices} are given block-entrywise by
  \begin{subequations} \label{eqn:dataFormulas}
  \begin{align} \label{eqn:MfromData}
    \bbLM_{\Bk, \Bj} & = -\frac{\varphi_{k} \varrho_{j}}{d(\i \vartheta_{k})
      d(\i \zeta_{j})} \frac{d(\i \vartheta_{k}) \BG(\i \vartheta_{k}) -
      d(\i \zeta_{j}) \left( \posTf(\i \zeta_{j}) +
      \frac{\vartheta_{k}}{\zeta_{j}}
      \velTf(\i \zeta_{j}) \right)}%
      {h(\i \vartheta_{k}) - h(\i \zeta_{j})},\\[2pt]
    \label{eqn:KfromData}
    \bbLK_{\Bk, \Bj} & = \hphantom{-} \frac{\varphi_{k} \varrho_{j}}%
      {d(\i \vartheta_{k})d(\i \zeta_{j})}
      \frac{n(\i \vartheta_{k}) \BG(\i \vartheta_{k}) - n(\i \zeta_{j})
      \left( \posTf(\i \zeta_{j}) + \frac{\vartheta_{k}}{\zeta_{j}}
      \velTf(\i \zeta_{j}) \right)}%
      {h(\i \vartheta_{k}) - h(\i \zeta_{j})},\\[2pt]
    \label{eqn:BCpCvfromData}
    \left( \bbBu \right)_{\Bk, :} & =\hphantom{-} \varphi_{k}
      \BG(\i \vartheta_{k}), \quad
    \left( \bbCp \right)_{:, \Bj} = \varrho_{j} \posTf(\i \zeta_{j})
    \quad\text{and}\quad
    \left( \bbCv \right)_{:, \Bj} = \frac{\varrho_{j}}{\i \zeta_{j}}
      \velTf(\i \zeta_{j}),
  \end{align}
  \end{subequations}
  where $\posTf$ and $\velTf$ are defined as in~\cref{eqn:pvTfs}.
\end{theorem}
\begin{proof} \allowdisplaybreaks
  For any indices $i, \ell \in \N$, we introduce the matrix
  \begin{equation} \label{eqn:grabSubMat}
    \BI_{i, \ell} = \begin{bmatrix} \Be_{(i - 1) \ell + 1} &
      \Be_{(i - 1) \ell + 2} & \cdots & \Be_{i \ell} \end{bmatrix}
      \in \R^{n \times \ell},
  \end{equation}
  where $\Be_{i} \in \Rn$ is the $i$-th canonical basis vector with $1$ at the
  $i$-th position and $0$ everywhere else.
  Note that $\BI_{i, \ell}$ contains a subset of the columns of the
  $n \times n$ identity and has the effect of extracting the columns
  $(i - 1) \ell + 1$ through $i \ell$ of a matrix by right multiplication.
  The formulae for $\bbBu$, $\bbCp$ and $\bbCv$ in~\cref{eqn:BCpCvfromData} are
  a direct consequence of their construction in~\cref{eqn:soDataMatrices} and
  the definitions of the quadrature-based factors $\BRcheckp$ and $\BLcheckv$.
  We easily observe that
  \begin{align*}
    \left( \bbBu \right)_{\Bk, :} = \BI_{k, p}^{\trans} \bbBu
      = \BI_{k, p}^{\trans} \BLcheckv^{\herm} \BBu
      = \varphi_{k} \left( \BCp + \i \vartheta_{k} \BCv \right)
        \Bvarphi(\i \vartheta_{k})^{-1} \BBu
      = \varphi_{k} \BG(\i \vartheta_{k})
  \end{align*}
  holds, where $\BI_{k,p}$ is defined according to~\cref{eqn:grabSubMat}. Similarly, we have that
  \begin{align*}
    \left( \bbCp \right)_{:, \Bj} & = \BCp \BRcheckp \BI_{j, m}
      = \varrho_{j} \BCp \Bvarphi(\i \zeta_{j})^{-1} \BBu
      = \varrho_{j} \posTf(\i \zeta_{j}) \quad \text{and}\\
    \left( \bbCv \right)_{:, \Bj} & = \BCv \BRcheckp \BI_{j, m}
      = \varrho_{j} \BCv \Bvarphi(\i \zeta_{j})^{-1} \BBu
      = \frac{\varrho_{j}}{\i \zeta_{j}} \velTf(\i \zeta_{j}).
  \end{align*}
  For the remaining formulae~\cref{eqn:MfromData,eqn:KfromData}, we first note
  that so long as $d(s) \neq 0$, it holds that
  \begin{align*}
    \Bvarphi(s)^{-1} & = \big( s^{2} \BM + s \BD(s) + \BK \big)^{-1} \\
    & = \Big( \big( s^{2} + s f(s) \big) \BM + \big( 1 + s g(s) \big)
      \BK \Big)^{-1}\\
    & = \frac{1}{1 + s g(s)} \left( \frac{s^{2} + s f(s)}{1 + s g(s)}
      \BM + \BK \right)^{-1}\\
    & = \frac{1}{d(s)} \big( h(s) \BM + \BK \big)^{-1}\\
    & = \frac{1}{d(s)} \Bphi \left( h(s) \right)^{-1},
  \end{align*}
  where $\Bphi(s) \defeq \left( s \BM + \BK \right)$.
  To prove~\cref{eqn:MfromData,eqn:KfromData}, we use the resolvent identities
  \cref{eqn:resolventId1} and~\cref{eqn:resolventId2}.
  Under the assumption that $h(\i \vartheta_{k}) \neq h(\i \zeta_{j})$ for all
  $k$ and $j$, applying~\cref{eqn:resolventId1} reveals that
  \begin{align*}
    \bbLM_{\Bk, \Bj} & = \BI_{k, p}^{\trans} \left( \BLcheckv^{\herm}
      \BM \BRcheckp \right) \BI_{j, m} \\
    & = \varphi_{k} \varrho_{j} \left( \BCp + \i \vartheta_{k} \BCv \right)
      \Bvarphi(\i \vartheta_{k})^{-1} \BM \Bvarphi(\i \zeta_{j})^{-1} \BBu \\
    & = \varphi_{k} \varrho_{j} \left( \BCp + \i \vartheta_{k} \BCv \right)
      \frac{\Bphi \left( h(\i \vartheta_{k}) \right)^{-1} \BM \Bphi
      \left( h(\i \zeta_{j}) \right)^{-1}}%
      {d(\i \vartheta_{k}) d(\i \zeta_{j})} \BBu \\
    & = -\varphi_{k} \varrho_{j} \left( \BCp + \i \vartheta_{k} \BCv \right)
      \frac{\Bphi \left( h(\i \vartheta_{k}) \right)^{-1} -
      \Bphi \left( h(\i \zeta_{j}) \right)^{-1}}%
      {d(\i \vartheta_k) d(\i \zeta_{j}) \left(h(\i \vartheta_{k}) -
      h(\i \zeta_{j}) \right)} \BBu \\
    & = -\frac{\varphi_{k} \varrho_{j}}%
      {d(\i \vartheta_{k}) d(\i \zeta_{j})}
      \frac{d(\i \vartheta_{k}) \BG(\i \vartheta_{k}) -
      d(\i \zeta_{j}) \left( \posTf(\i \zeta_{j}) +
      \frac{\vartheta_{k}}{\zeta_{j}} \velTf(\i \zeta_{j}) \right)}%
      {h(\i \vartheta_{k}) - h(\i \zeta_{j})},
  \end{align*}
  where the last line follows from the relationship
  $\Bvarphi(s)^{-1} = d(s)^{-1} \Bphi \left( h(s) \right)^{-1}$ and the
  definitions of $\posTf$ and $\velTf$ in~\cref{eqn:pvTfs}.
  Likewise, applying the second resolvent identity~\cref{eqn:resolventId2}
  reveals that
  \begin{align*}
    \bbLK_{\Bk, \Bj} & = \BI_{k, p}^{\trans} \left( \BLcheckv^{\herm}
      \BK \BRcheckp \right) \BI_{j, m} \\
    & = \varphi_{k} \varrho_{j} \left( \BCp + \i \vartheta_{k} \BCv \right)
      \Bvarphi(\i \vartheta_{k})^{-1} \BK \Bvarphi(\i \zeta_{j})^{-1} \BBu \\
    & = \varphi_{k} \varrho_{j} \left( \BCp + \i \vartheta_{k} \BCv \right)
      \frac{\Bphi \left( h(\i \vartheta_{k}) \right)^{-1} \BK \Bphi
      \left( h(\i \zeta_{j}) \right)^{-1}}%
      {d(\i \vartheta_{k}) d(\i \zeta_{j})} \BBu \\
    & = \varphi_{k} \varrho_{j} \left( \BCp + \i \vartheta_{k} \BCv \right)
      \frac{h(\i \vartheta_{k}) \Bphi \left( h(\i \vartheta_{k}) \right)^{-1} -
      h(\i \zeta_{j}) \Bphi \left( h(\i \zeta_{j}) \right)^{-1}}%
      {d(\i \vartheta_{k}) d(\i \zeta_{j}) \left( h(\i \vartheta_{k}) -
      h(\i \zeta_{j}) \right)} \BBu \\
    & = \frac{\varphi_{k} \varrho_{j}}%
      {d(\i \vartheta_{k}) d(\i \zeta_{j})}
      \frac{n(\i \vartheta_{k}) \BG(\i \vartheta_{k}) - n(\i \zeta_j)
      \left( \posTf(\i \zeta_{j}) + \frac{\vartheta_{k}}{\zeta_{j}}
      \velTf(\i \zeta_{j}) \right)}%
      {h(\i \vartheta_{k}) - h(\i \zeta_{j})},
  \end{align*}
  thus proving the final remaining formula~\cref{eqn:KfromData}.
\end{proof}

In the formulae~\cref{eqn:MfromData,eqn:KfromData}, the appearing term
$\posTf(\i \zeta_{j}) + (\vartheta_{k} / \zeta_{j}) \velTf(\i \zeta_{j})$
looks contrived at first glance. 
However, this is an artifact of \Cref{thm:sobtFromData} being stated as
generally as possible to include position and velocity outputs simultaneously;
neither is necessarily required.
In realistic applications, it is usually the case that either $\BCp$ or $\BCv$
is identically zero.
Then, the term will resolve to
\begin{equation*}
  \posTf(\i \zeta_{j}) + \frac{\vartheta_{k}}{\zeta_{j}}
    \velTf(\i \zeta_{j}) = \posTf(\i \zeta_{j}) = \BG(\i\zeta_{j})
  \quad\text{if}\quad
  \BCv = \Bzero,
\end{equation*}
which is exactly the second-order transfer function~\cref{eqn:soTf}, or to
\begin{equation*}
  \posTf(\i \zeta_{j}) + \frac{\vartheta_{k}}{\zeta_{j}} \velTf(\i \zeta_{j})
    = \frac{\vartheta_{k}}{\zeta_{j}} \velTf(\i \zeta_{j})
    = \frac{\vartheta_{k}}{\zeta_{j}} \BG(\i \zeta_{j})
  \quad\text{if}\quad
  \BCp = \Bzero,
\end{equation*}
which is a re-scaling of the second-order transfer function~\cref{eqn:soTf}.

Replacing the intrusive quantities~\cref{eqn:soBTKeyQuantities} in
\Cref{alg:pvBT} with the data-based approximations~\cref{eqn:soDataMatrices}
results in a data-driven reformulation of \sopvBt{} summarized in
\Cref{alg:soQuadBT}, which we call \emph{second-order quadrature-based
position-velocity balanced truncation} (\soQuadbt).
Outside of empirical knowledge used to design the damping coefficient functions
$f$ and $g$, the method is entirely non-intrusive.
As in~\cite{GosGB22}, the quadrature-based approximations to the Gramians are
never explicitly formed.
Instead, they are only invoked implicitly to derive the data-based
approximations~\cref{eqn:soDataMatrices}; the points in~\Cref{alg:soQuadBT} at which $\posTf$ and $\velTf$ are sampled are the quadrature nodes underlying these quadrature rules.
While \Cref{alg:soQuadBT} is formulated generally to include position and
velocity outputs, as is the case with \Cref{thm:sobtFromData}, both are not
necessarily required; depending on the application of interest, only one of
them is required.
Lastly, \Cref{thm:sobtFromData} assumes that the sets of quadrature nodes used
to implicitly approximate $\BPp$ and $\BQv$ are \emph{distinct}.
In \Cref{thm:sobtFromData_Hermite}, we derive a special case of the
formulae~\cref{eqn:dataFormulas} when the quadrature nodes
come in complex conjugate pairs, i.e., $\i \zeta_{j} = \i s_{j}$ and
$\i \vartheta_{j} = -\i s_{j}$ for all $j=1, \ldots, J$, and may overlap.

\begin{algorithm}[t!]
  \SetAlgoHangIndent{1pt}
  \DontPrintSemicolon
  \caption{Quadrature-based \sopvBt{} (\soQuadbt{}).} 
  \label{alg:soQuadBT}

  \KwIn{System transfer function~$\tf$,
    position and velocity mappings $\posTf$ and $\velTf$,
    damping coefficients $f$ and $g$,
    left and right quadrature nodes and weights
    $\{(\i \vartheta_{k}, \varphi_{k})\}_{k = 1}^{K}$ and
    $\{(\i \zeta_{j}, \varrho_{j}) \}_{j = 1}^{J}$,
    reduction order $r$.}
  \vspace{2pt}
  \KwOut{Reduced-order system matrices $\BMr, \BDr, \BKr, \BBur, \BCpr, \BCvr$
    in~\cref{eqn:soSysRed}.}

  Evaluate the mappings at the quadrature nodes to obtain the data
    \begin{align*}
      & \left\{ \Big( \posTf(\i \zeta_{j}), \velTf(\i \zeta_{j}),
        f(\i \zeta_{j}), g(\i \zeta_{j}) \Big) \right\}_{j = 1}^{J}
        \quad\text{and} \\
      & \left\{ \Big( \tf(\vartheta_{k}), f(\i \vartheta_k),
        g(\i \vartheta_{k}) \Big) \right\}_{k = 1}^{K},
    \end{align*}
    and construct the data matrices $\bbLM, \bbLK, \bbBu, \bbCp, \bbCv$
    according to \Cref{thm:sobtFromData}.\;
  
  Compute the singular value decomposition
    \begin{equation*}
      \bbLM = \begin{bmatrix} \BUcheck_{1} & \BUcheck_{2}\end{bmatrix}
        \begin{bmatrix} \BSigmacheck_{1} & \Bzero \\ \Bzero & \BSigmacheck_{2}
        \end{bmatrix}
        \begin{bmatrix} \BYcheck_{1}^{\herm} \\[2pt] \BYcheck_{2}^{\herm}
        \end{bmatrix},
    \end{equation*}
    for $\BSigmacheck_{1} \in \Rrr$ diagonal containing the $r$ largest
    singular values, $\BSigmacheck_{2} \in \R^{(pK - r) \times (mJ - r)}$
    diagonal and $\BUcheck_{1}, \BUcheck_{2}, \BYcheck_{1}, \BYcheck_{2}$
    partitioned accordingly.\;
    
  Compute the reduced-order model matrices according to
    \begin{align*}
      \BMr & = \BI_{r}, \\
      \BKr & = \BSigmacheck_1^{-1/2} \BUcheck_{1}^{\herm} \bbLK
        \BYcheck_{1} \BSigmacheck_{1}^{-1/2}, \\
      \BDr(s) & = f(s) \BI_{r} + g(s) \BKr, \\
      \BBur & = \BSigmacheck_{1}^{-1/2} \BUcheck_{1}^{\herm} \bbBu, \\
      \BCpr & = \bbCp \BYcheck_{1} \BSigmacheck_{1}^{-1/2}, \\
      \BCvr & = \bbCv \BYcheck_{1} \BSigmacheck_{1}^{-1/2}.
    \end{align*}
\end{algorithm}

\begin{remark} \label{remark:soLoewner}
  It is worth highlighting that the matrices in~\cref{eqn:soDataMatrices} are
  highly similar to the second-order Loewner and shifted Loewner matrices that
  appear in the extension of the interpolatory Loewner framework for
  second-order systems developed in~\cite{PonGB22}. 
  Suppose that $\BCv = \Bzero$ and that the system~\cref{eqn:soSys} has
  Rayleigh damping~\cref{eqn:Rayleigh}.
  Given two disjoint sets of interpolation points
  $\lambda_{1}, \ldots, \lambda_{K} \in \C$ and
  $\mu_{1}, \ldots, \mu_{J} \in \C$, the second-order Loewner and
  shifted Loewner matrices $\bbL \in \C^{pK \times mJ}$ and
  $\bbLs \in \C^{pK \times mJ}$ are defined block-elementwise as
  \begin{equation} \label{eqn:soLoewner}
    \bbL_{\Bk, \Bj} = \frac{d(\lambda_{k}) \BG(\lambda_{k}) -
      d(\mu_{j}) \BG(\mu_{j})}{h(\lambda_{k}) - h(\mu_{j})}, \quad
    \bbLs_{\Bk, \Bj} = \frac{n(\lambda_{k}) \BG(\lambda_{k}) -
      n(\mu_{j}) \BG(\mu_{j})}{h(\lambda_{k}) - h(\mu_{j})},
  \end{equation}
  where $n$, $d$ and $h$ are defined as in~\cref{eqn:auxFunctions} for
  $f(s) = \alpha$ and $g(s) = \beta$; see~\cite[Def.~2]{PonGB22}.
  If $\BCp = \Bzero$, and we take $\mu_{j} = \i \zeta_{j}$ and
  $\lambda_{k} = \i \vartheta_{k}$, then the
  matrices~\cref{eqn:soDataMatrices}, which appear in \soQuadbt{}, are related
  to the second-order Loewner and shifted Loewner matrices~\cref{eqn:soLoewner}
  by diagonal scalings:
  \begin{align*}
    \bbLM & = -\diag \left( \varphi_{1} d(\i \vartheta_{1})^{-1}, \ldots,
      \varphi_{K} d(\i \vartheta_{K})^{-1} \right) \bbL
      \diag \left( \varrho_{1} d(\i \zeta_{1})^{-1}, \ldots,
      \varrho_{J} d(\i \zeta_{J})^{-1} \right),\\
    \bbLK & = \hphantom{-} \diag \left( \varphi_{1} d(\i \vartheta_{1})^{-1},
      \ldots, \varphi_{K} d(\i \vartheta_{K})^{-1} \right) \bbLs
      \diag \left( \varrho_{1} d(\i \zeta_{1})^{-1}, \ldots,
      \varrho_{J} d(\i \zeta_{J})^{-1} \right).
  \end{align*}
  In~\cite{PonGB22}, the matrices~\cref{eqn:soLoewner} are used to construct
  rational interpolants of second-order systems from data;
  see~\cite[Thm.~3]{PonGB22} for further details.
  The input and output matrices of this rational interpolant are those
  in~\cref{eqn:BCpCvfromData} multiplied from the left and right, respectively,
  by the diagonal scalings given above.
  We also highlight that the derivation of the diagonally scaled Loewner
  matrices in the proof of \Cref{thm:sobtFromData} is different from that
  of~\cite[Thm.~3]{PonGB22}, which uses a parameterization of interpolating
  systems~\cref{eqn:soSysRed} based on matrix equations.
  The tools and techniques used in the proof of \Cref{thm:sobtFromData} can be
  applied to derive this generalization of the Loewner framework
  from~\cite{PonGB22}, as well.
  Significantly, this extends the applicability of the framework to
  second-order systems~\cref{eqn:soSys} with the generalized proportional
  damping model~\cref{eqn:dampingModel}, as well as mixed position and
  velocity outputs.
\end{remark}


\section{Theoretical considerations}%
\label{sec:theorConsiderations}

In this section, we consider theoretical implications of the
\soQuadbt{} framework including the preservation of asymptotic stability and
the error between the data-driven terms and their intrusive counterparts.


\subsection{Stability preservation}%
\label{sec:stability}

For the present discussion, we allow the matrices in~\cref{eqn:soSys} to be
complex valued.
In the intrusive setting, \sopvBt{} preserves asymptotic stability when the
underlying full-order model~\cref{eqn:soSys} is \emph{Hermitian}, i.e.,
\begin{equation} \label{eqn:soSymmReal}
  \BM = \BM^{\herm}, \quad
  \BD = \BD^{\herm}, \quad
  \BK = \BK^{\herm}, \quad
  \BBu = \BCp^{\herm} \quad\text{and}\quad
  \BCv = \Bzero,
\end{equation}
and the system matrices $\BM, \BD, \BK$ are Hermitian positive definite
(\HPD{}).
Obviously, a special case of~\cref{eqn:soSymmReal} is the
system~\cref{eqn:soSys} being state-space symmetric when the corresponding
matrices are real-valued.
Interestingly, \soQuadbt{} will preserve asymptotic stability if the data are
generated by a Rayleigh-damped system that is Hermitian with \HPD{} mass,
damping and stiffness matrices, and if the left and right quadrature
nodes---those used to implicitly approximate $\BQv$ and $\BPp$,
respectively---are complex conjugate pairs of each other, i.e.,
$\i \zeta_{i} = \i s_{i}$ and $\i \vartheta_{i} = -\i s_{i}$.
If the $s_{i}$'s are not all strictly positive or negative, then at least some
of the nodes will be equivalent, i.e.,
$\zeta_{j} = s_{j} = - s_{k} = \vartheta_{k}$ for some $j, k$, which is not
allowed under~\Cref{thm:sobtFromData}.
To prove the stability result, we first derive formulae for the data
matrices~\cref{eqn:soDataMatrices} in this special case when the quadrature
nodes underlying the approximations to $\BQv$ and $\BPp$ come in complex
conjugate pairs and may overlap, with $\BCv = \Bzero$. 

\begin{theorem} \label{thm:sobtFromData_Hermite} \allowdisplaybreaks
  Let the functions $d \colon \C \to \C$, $n \colon \C \to \C$ and
  $h \colon \C \to \C$ be defined as in~\cref{eqn:auxFunctions}.
  Suppose the left and right quadrature rules
  in~\cref{eqn:soQuadContFactor,eqn:soQuadObsvFactor} are complex conjugate
  pairs, i.e., $\zeta_{i} = s_{i}$, $\vartheta_{i} = -s_{i}$ and
  $\xi_{i} = \varrho_{i} = \vartheta_{i}$ for $i = 1, \ldots, N$ and
  $N = J = K$, and the quadrature rules are such that $d(\i s_i)\neq 0$ and $\Bvarphi(\i s_{j})$ is nonsingular for all $i = 1, \ldots, N$.
  Let the quadrature-based factors $\BRcheckp$ and $\BLcheckv$ be given as
  in~\cref{eqn:soQuadContFactor,eqn:soQuadObsvFactor} and let the matrices
  $\bbLM, \bbLK \in \C^{pK \times mJ}$, $\bbBu\in\C^{pK \times m}$ and $\bbCp\in\C^{p \times mJ}$ be as in~\cref{eqn:soDataMatrices}.
  Then, the matrices $\bbLM$, $\bbLK$, $\bbBu$ and $\bbCp$ are given
  block-entrywise by
  \begin{subequations} \label{eqn:dataFormulas_Hermite}
  \begin{align}
    \label{eqn:MfromData_Hermite}
    \bbLM_{\Bk, \Bj} & =
      \begin{cases}
        \displaystyle -\frac{\xi_{k} \xi_{j}}{d(\i s_{j})^2}
          \frac{\ds}{\ds s} \Big( d(s) \BG(s) \Big) \bigg|_{s = \i s_{j}}
          \left( \frac{\ds}{\ds s} h(s) \bigg|_{s = \i s_{j}} \right)^{-1} &
          \text{if} \quad \i s_{j} = -\i s_{k}, \\[5mm]
        \displaystyle -\frac{\xi_{k} \xi_{j}}{d(-\i s_{k}) d(\i s_{j})}
          \frac{d(-\i s_{k})\BG(-\i s_{k}) - d(\i s_{j}) \BG(\i s_{j})}%
          {h(-\i s_{k}) - h(\i s_{j})} & \text{if} \quad \i s_{j} \neq -\i s_{k},
      \end{cases}\\[2ex]
    \label{eqn:KfromData_Hermite}
    \bbLK_{\Bk, \Bj} & =
      \begin{cases}
        \displaystyle \frac{\xi_{k} \xi_{j}}{d(\i s_{j})^2}
          \frac{\ds}{\ds s} \Big( n(s) \BG(s) \Big) \bigg|_{s = \i s_{j}} 
          \left( \frac{\ds}{\ds s} h(s) \bigg|_{s = \i s_{j}} \right)^{-1} &
          \text{if} \quad \i s_{j} = -\i s_{k}, \\[5mm]
    \displaystyle \frac{\xi_{k} \xi_{j}}{d(-\i s_{k}) d(\i s_{j})}
      \frac{n(-\i s_{k}) \BG(-\i s_{k}) - n(\i s_{j}) \BG(\i s_{j})}%
      {h(-\i s_{k}) - h(\i s_{j})} & \text{if} \quad \i s_{j} \neq -\i s_{k},
      \end{cases}\\[2ex]
    \label{eqn:BCpCvfromData_Hermite}
    \big( \bbBu & \big)_{\Bk, :} = \xi_k \BG(-\i s_{k})
    \quad\text{and}\quad
    \big( \bbCp \big)_{:, \Bj} = \xi_j \BG(\i s_{j}),
  \end{align}
  \end{subequations}
  where $\BG$ is the second-order transfer function~\cref{eqn:soTf}.
\end{theorem}
\begin{proof}  \allowdisplaybreaks
  The claims in~\cref{eqn:BCpCvfromData_Hermite}, as well
  as~\cref{eqn:MfromData_Hermite,eqn:KfromData_Hermite} for the case of
  $\i s_{j} \neq -\i s_{k}$ follow directly from \Cref{thm:sobtFromData} under the stated
  assumptions since $\zeta_{j} = s_{j}$ and $\vartheta_{k} = -s_{k}$.
  To prove~\cref{eqn:MfromData_Hermite} in the $\i s_{j} = -\i s_{k}$ case, we observe that
  \begin{equation*}
    \bbLM_{\Bk, \Bj} = \BI_{k, p}^{\trans} \left( \BLcheckv^{\herm}
      \BM \BRcheckp \right) \BI_{j, m} =
      \xi_{k} \xi_{j} \BCp \Bvarphi(\i s_{j})^{-1} \BM \Bvarphi(\i s_{j})^{-1} \BBu
  \end{equation*}
  holds.
  Let $\varepsilon > 0$; using the same logic as in the proof of
  \Cref{thm:sobtFromData}, we have that
  \begin{align*}
    & \BCp \Bvarphi(\i s_{j})^{-1} \BM \Bvarphi(\i s_{j})^{-1} \BBu \\
    & = \lim\limits_{\varepsilon \to 0} \BCp
      \Bvarphi \big( \i (s_{j} + \varepsilon) \big)^{-1} \BM
      \Bvarphi(\i s_{j})^{-1} \\
    & = -\lim\limits_{\varepsilon \to 0}
      \frac{1}{d \big(\i (s_{j} + \varepsilon) \big) d(\i s_{j})}
      \frac{d \big( \i(s_{j} + \varepsilon) \big) \BG
      \big( \i (s_{j} + \varepsilon) \big) - d(\i s_{j}) \BG(\i s_{j})}%
      {h \big(\i (s_{j} + \varepsilon) \big) - h(\i s_{j})} \\
    & = -\frac{1}{d(\i s_{j})^{2}} \lim\limits_{\varepsilon \to 0}
      \frac{d \big( \i(s_{j} + \varepsilon) \big) \BG
      \big(\i (s_{j} + \varepsilon) \big) - d(\i s_{j}) \BG(\i s_{j})}%
      {\i (s_{j} + \varepsilon) - \i s_{j}}
      \frac{\i (s_{j} + \varepsilon) - \i s_{j}}%
      {h \big(\i (s_{j} +\varepsilon) \big) - h(\i s_{j})}\\
    & = -\frac{1}{d(\i s_{j})^{2}} \frac{\ds}{\ds s}
      \Big( d(s) \BG(s) \Big) \bigg|_{s = \i s_{j}}
      \left( \frac{\ds}{\ds s} h(s) \bigg|_{s = \i s_{j}} \right)^{-1},
  \end{align*}
  by the limit definition of the derivative.
  This proves~\cref{eqn:MfromData_Hermite} when $\i s_{j} = -\i s_{k}$ by
  multiplication with $\xi_{k} \xi_{j}$.
  Using the same arguments, the other remaining
  formula~\cref{eqn:KfromData_Hermite} can be shown for $\i s_{j} = -\i s_{k}$, which
  concludes the proof.
\end{proof}

\Cref{alg:soQuadBT} can be used with the setup and assumptions outlined in
\Cref{thm:sobtFromData_Hermite} to compute reduced-order models via \soQuadbt{}.
The substantial difference to using \Cref{thm:sobtFromData} is that
the data matrices in~\cref{eqn:soDataMatrices} are constructed according to the
formulae in~\cref{eqn:dataFormulas_Hermite} instead of~\cref{eqn:dataFormulas}. 
In this case, the matrices in~\cref{eqn:soDataMatrices} may require information
about the derivative of $\BG$ in the off-diagonal entries if the chosen
quadrature nodes are such that $\i s_{j} = -\i s_{k}$.
This will indeed be the case for the realification procedure that we describe
in \Cref{sec:realmodel}.

Suppose the assumptions of \Cref{thm:sobtFromData_Hermite} are satisfied and
that the transfer function data $\BG(\i s_{j})$ are generated by a Hermitian
system~\cref{eqn:soSymmReal} with \HPD{} mass, damping and stiffness matrices. 
Because the underlying system is Hermitian, it follows that the
quadrature-based factors~\cref{eqn:soQuadContFactor,eqn:soQuadObsvFactor}
satisfy $\BRcheckp = \BLcheckv$ since $\i \zeta_{j} = \i s_{j}$ and
$\i \vartheta_{k} = -\i s_{k}$.
Thus, the matrices $\bbLM$ and $\bbLK$ constructed according to
\Cref{thm:sobtFromData_Hermite} satisfy
\begin{equation*}
  \bbLM = \BRcheckp^{\herm}\BM\BRcheckp = \bbLM^{\herm} \quad\text{and}\quad
  \bbLK = \BRcheckp^{\herm}\BK\BRcheckp = \bbLK^{\herm}
\end{equation*}
by the property~\cref{eqn:soSymmReal} of the underlying system.
Moreover, $\bbBu = \bbCp^{\herm}$ holds by construction.
We claim that a reduced-order model~\cref{eqn:soSysRed} computed via
\Cref{alg:soQuadBT} using these Hermitian data matrices has a similar symmetry
structure as the original underlying system with \HPD{} mass, damping and
stiffness matrices.
Obviously $\BMr = \BI_r$ is \HPD{} by construction.
Because $\bbLM$ is Hermitian, its singular value decomposition is such that the
singular vectors satisfy $\BUcheck = \BYcheck$. 
Then, for any $\Bz\in\C^{Nm}$, it holds that
\begin{equation*}
  \Bz^{\herm} \BKr \Bz = \Bz^{\herm}\BSigmacheck_{1}^{-1/2} \BUcheck_{1}^{\herm}
    \bbLK \BUcheck_{1} \BSigmacheck_{1}^{-1/2} \Bz =
    \Bw^{\herm} \BK \Bw > 0,
\end{equation*}
with $\Bw = \BRcheckp \BUcheck_{1} \BSigmacheck_{1}^{-1/2} \Bz$, and the fact that
$\Bw^{\herm} \BK \Bw > 0$ follows from the assumption that $\BK$ is \HPD{}.
Therefore, we have that $\BKr$ is \HPD{}.
Finally, if the system is Rayleigh damped, then trivially $\BDr$ is \HPD{}
because it is a nonnegative linear combination of two \HPD{} matrices and
$\alpha, \beta \geq 0$ by assumption.
Because Hermitian second-order systems~\cref{eqn:soSys} with positive definite
mass, damping, and stiffness matrices are asymptotically stable,
this argument directly implies the following stability result.

\begin{theorem} \label{thm:stability}
  Assume the system $\Sys$ in~\cref{eqn:soSys} used to generate the data
  $\BG(\i s_j)$ is Hermitian~\cref{eqn:soSymmReal} with positive definite mass,
  damping and stiffness matrices, assume that $\Sys$ has Rayleigh
  damping~\cref{eqn:Rayleigh}, and assume that the data $\BG(\i s_j)$
  satisfy the assumptions of \Cref{thm:sobtFromData_Hermite}.
  Then, the reduced-order system computed by \soQuadbt{} using the data
  matrices from \Cref{thm:sobtFromData_Hermite} is asymptotically stable. 
\end{theorem}

\Cref{thm:stability} may also be used conversely to check the assumptions
placed on the underlying system that provides the data.
If a reduced-order model~\cref{eqn:soSysRed} is computed by \Cref{alg:soQuadBT}
using the data matrices~\cref{eqn:dataFormulas_Hermite} is not asymptotically
stable, then the transfer function data $\BG(\i s_{j})$ were not generated by a
symmetric system with \HPD{} mass, damping and stiffness matrices.


\subsection{Error control}%
\label{sec:error}

The result~\cite[Prop.~3.2]{GosGB22} provides error control on the data
matrices that appear in \foQuadbt{} in terms of their intrusive counterparts. 
By recognizing that the proof of~\cite[Prop.~3.2]{GosGB22} makes no assumption
on the underlying system structure, this result can be extended directly to our
setting to derive error bounds on the matrices in \Cref{thm:sobtFromData}.

\begin{theorem} \label{thm:error} \allowdisplaybreaks
  Assume that the quadrature rules used to
  derive~\cref{eqn:soQuadContFactor,eqn:soQuadObsvFactor} are such that the
  quadrature-based approximations to $\BPp$ and $\BQv$ satisfy
  \begin{equation} \label{eqn:quadControl}
    \lVert \BPp - \BRcheckp \BRcheckp^{\herm} \rVert_{\frob} \leq
      \frac{\delta}{1 + \delta} \sigma_{\min} \left( \BPp \right)
    \quad\text{and}\quad
    \lVert \BQv - \BLcheckv \BLcheckv^{\herm} \rVert_{\frob} \leq
      \frac{\delta}{1 + \delta} \sigma_{\min} \left( \BQv \right),
  \end{equation}
  for some $\delta \in (0, 1)$, where $\sigma_{\min}(\cdot)$ denotes the minimum
  singular value of a matrix.
  Then, there exist partial isometries $\leftIso \in \C^{pK \times n}$ and
  $\rightIso \in \C^{mJ \times n}$ so that
  \begin{subequations} \label{eqn:romBounds}
  \begin{align}
    \lVert \BLcheckv^{\herm} \BM \BRcheckp - \leftIso \big( \BLv^{\herm}
      \BM \BRp \big) \rightIso^{\herm} \rVert_{\frob} & \leq 2 \delta
      \lVert \BLv \rVert_{2} \lVert \BM \rVert_{2} \lVert \BRp \rVert_{2}, \\
    \lVert \BLcheckv^{\herm} \BK \BRcheckp - \leftIso \big( \BLv^{\herm}
      \BK \BRp \big) \rightIso^{\herm} \rVert_{\frob} & \leq 2 \delta
      \lVert \BLv \rVert_{2} \lVert \BK \rVert_{2} \lVert \BRp \rVert_{2}, \\
    \lVert \BLcheckv^{\herm} \BB - \leftIso \big( \BLv^{\herm} \BB \big)
      \rVert_{\frob} & \leq 2 \delta \lVert \BLv \rVert_{2}
      \lVert \BB \rVert_{2}, \\
    \lVert \BCp \BRcheckp - \big( \BCp \BRp \big) \rightIso^{\herm}
      \rVert_{\frob} & \leq 2 \delta \lVert \BCp \rVert_{2}
      \lVert \BRp \rVert_{2}, \\
    \lVert \BCv \BRcheckp - \big( \BCv \BRp \big) \rightIso^{\herm}
      \rVert_{\frob} & \leq 2 \delta \lVert \BCv \rVert_{2}
      \lVert \BRp \rVert_{2}.
  \end{align}
  \end{subequations}
\end{theorem}
\begin{proof}
  Note that the proof of~\cite[Prop.~3.2]{GosGB22} makes no
  assumptions on the underlying system structure.
  Thus, the inequalities in~\cref{eqn:romBounds} follow from carefully
  executing the proof of~\cite[Prop.~3.2]{GosGB22} in the context of
  \Cref{thm:sobtFromData}.
\end{proof}

In general, \Cref{thm:error} can be applied to any low-rank implementation of
\sopvBt{} for which the approximate factors satisfy~\cref{eqn:quadControl}.
For our setting, this allows us to control the error in the reduced-order model
matrices via the error in the corresponding quadrature rules. 
Thus, for a sufficiently fine resolved quadrature rule, one can recover nearly
exact \sopvBt{}-based reduced-order models solely from data.


\section{Practical considerations}%
\label{sec:pracConsiderations}

In this section, we discuss practical considerations for the implementation of
\Cref{alg:soQuadBT}.
First, we provide formulas for the \soQuadbt{} method to construct real-valued
reduced-order models, and then we discuss how to infer suitable damping
parameters for the proportional damping model from data.


\subsection{Enforcing real-valued models from complex-valued data}%
\label{sec:realmodel}

In the context of systems used in  time-domain simulations, the underlying
system~\cref{eqn:soSys} is typically presumed to be real in the sense that the
inputs, state variables, and outputs evolve in a real-dimensional vector space. 
Thus, there exists a state-space realization of~\cref{eqn:soSys} given by
\emph{real-valued} matrices $\BM$, $\BD$, $\BK$, $\BBu$, $\BCp$ and $\BCv$. 
When such a real system generates the transfer function data used in
\Cref{alg:soQuadBT}, it is possible to modify \soQuadbt{} to produce a
real-valued realization of the reduced-order model.
In the subsequent discussion, we describe how to achieve the construction of
such real-valued realizations.
To this end, we require that the frequency-dependent damping coefficients $f$
and $g$ in~\cref{eqn:dampingModel} to satisfy
\begin{subequations} \label{eqn:conjAssumption}
\begin{align}
  \overline{f(s)} & = f(\overline{s}), \quad
  \overline{g(s)} = g(\overline{s}) 
  \quad\text{so that}\quad \\
  \overline{n(s)} & = n(\overline{s}), \quad
  \overline{d(s)} = d(\overline{s}), \quad
  \overline{h(s)} = h(\overline{s}),
\end{align}
\end{subequations}
for all $s\in \C$, where $n$, $d$ and $h$ are defined as
in~\cref{eqn:auxFunctions}. 
The assumption~\cref{eqn:conjAssumption} is trivially satisfied for Rayleigh
damping~\cref{eqn:Rayleigh} because $\alpha, \beta \in \R$, and notably
not satisfied for structural damping~\cref{eqn:structural} because
$\overline{g(s)} = -\i \tfrac{\overline{\eta}}{\overline{s}} \neq \i
\tfrac{\overline{\eta}}{\overline{s}} = g(\overline{s})$.
Although this is expected; indeed, a structurally-damped system can always be
expressed by a realization with zero damping and complex-valued stiffness
$\left(1 + \i \eta \right) \BK$.

For simplicity of presentation, we assume that both quadrature rules use an
equal number of nodes, $N = J = K$, and that $\BCv = \Bzero$;
the construction that we describe in the following also applies to the general case.
Suppose that a \emph{symmetric} quadrature rule is used to generate the
approximate factors~\cref{eqn:soQuadContFactor,eqn:soQuadObsvFactor}, i.e.,
the nodes are symmetrically distributed according to
\begin{subequations} \label{eqn:symmNodes}
\begin{align}
  \vartheta_{1} & < \vartheta_{3} < \ldots < \vartheta_{N - 1} < 0 <
    \vartheta_{2} < \vartheta_{4} < \ldots < \vartheta_{N}, \\
  \zeta_{1} & < \zeta_{3} < \ldots < \zeta_{N-1} < 0 <
    \zeta_{2} < \zeta_{4} < \ldots < \zeta_{N},
\end{align}
\end{subequations}
and satisfy $\vartheta_{i} = -\vartheta_{i+1}$ and $\zeta_{i} = -\zeta_{i+1}$,
for $i = 1, 3, \ldots, N-1$.
We assume additionally that the corresponding quadrature weights are computed
such that $\varrho_{i} = \varrho_{i + 1}$ and $\varphi_{i} = \varphi_{i + 1}$
for all $i = 1, 3, \ldots, N-1$.
This symmetric ordering scheme ensures that the transfer function evaluations
at these nodes come in complex conjugate pairs, that is,
$\overline{\BG(\i \vartheta_i)}=\BG(\i\vartheta_{i + 1})$ and
$\overline{\BG(\i \zeta_{i})} = \BG(\i \zeta_{i + 1})$,
for $i = 1, 3\ldots, N - 1$.
From this conjugate symmetry and the assumption~\cref{eqn:conjAssumption}, the
$2p \times 2m$-block entries of $\bbLM$ corresponding to the nodes
$(\i\vartheta_{k}, -\i \vartheta_{k})$ and $(\i \zeta_{j}, -\i\zeta_{j})$
satisfy
\begin{align*}
  & \begin{bmatrix} \bbLM_{\Bk, \Bj} & \bbLM_{\Bk, \Bj + \Bone} \\
    \bbLM_{\Bk + \Bone, \Bj} & \bbLM_{\Bk + \Bone, \Bj + \Bone} \end{bmatrix} \\
  & = -\varphi_{k} \varrho_{j}
    \begin{bmatrix}
    \displaystyle
    \frac{d(\i \vartheta_{k}) \BG(\i \vartheta_{k}) -
    d(\i \zeta_{j}) \BG(\i \zeta_{j})}%
    {d(\i \vartheta_{k}) d(\i \zeta_{j}) \left( h(\i \vartheta_{k}) -
    h(\i \zeta_{j}) \right)} &
    \displaystyle
    \frac{d(\i \vartheta_{k}) \BG(\i \vartheta_{k}) -
    \overline{d(\i \zeta_{j}) \BG(\i \zeta_{j})}}%
    {d(\i \vartheta_{k}) \overline{d(\i \zeta_{j})} \left( h(\i \vartheta_{k}) -
    \overline{h(\i \zeta_{j})} \right)} \\[3pt]
    \displaystyle
    \frac{\overline{d(\i \vartheta_{k}) \BG(\i \vartheta_{k})} -
    d(\i \zeta_{j}) \BG(\i \zeta_{j})}%
    {\overline{d(\i \vartheta_{k})} d(\i \zeta_{j}) \left(
    \overline{h(\i \vartheta_{k})} - h(\i \zeta_{j}) \right)} & 
    \displaystyle 
    \frac{\overline{d(\i \vartheta_{k})} \overline{\BG(\i \vartheta_{k})} -
    \overline{d(\i \zeta_{j})} \overline{\BG(\i \zeta_{j})}}%
    {\overline{d(\i \vartheta_{k})} \overline{d(\i \zeta_{j})}
    \left( \overline{h(\i \vartheta_{k})} - \overline{h(\i \zeta_{j})} \right)}
    \end{bmatrix},
\end{align*}
where $\Bk$ and $\Bk + \Bone$ refer to the indices $(k - 1) p + 1 : k p$ and
$k p + 1 : (k + 1) p$ and similarly for $\Bj$ and $\Bj + \Bone$.
From the above, we observe that
$\overline{\bbLM_{\Bk, \Bj}} = \bbLM_{\Bk + \Bone, \Bj + \Bone}$ and
$\overline{\bbLM_{\Bk, \Bj + \Bone}} = \bbLM_{\Bk + \Bone, \Bj}$.
Define the unitary matrix $\BJ_{\ell} \defeq \frac{1}{\sqrt{2}}
\begin{bmatrix} \BI_{\ell} & -\i \BI_{\ell} \\ \BI_{\ell} &  \hphantom{-}\i
\BI_{\ell} \end{bmatrix} \in \C^{2 \ell \times 2 \ell}$ for some positive
integer $\ell \in \N$.
One can readily verify that
\begin{equation*}
  \BJ_{p}^{\herm} \begin{bmatrix} \bbLM_{\Bk, \Bj} & \bbLM_{\Bk, \Bj + \Bone} \\
    \bbLM_{\Bk + \Bone, \Bj} & \bbLM_{\Bk + \Bone, \Bj + \Bone} \end{bmatrix}
    \BJ_{m} = \begin{bmatrix} \hphantom{-}\real(\bbLM_{\Bk, \Bj} +
    \bbLM_{\Bk, \Bj + \Bone}) & \imag(\bbLM_{\Bk, \Bj} -
    \bbLM_{\Bk, \Bj + \Bone}) \\ 
    -\imag(\bbLM_{\Bk, \Bj} + \bbLM_{\Bk, \Bj + \Bone}) &
    \real(\bbLM_{\Bk, \Bj} - \bbLM_{\Bk, \Bj + \Bone} ) \end{bmatrix},
\end{equation*}
which is real valued. 
Thus, the matrix $\bbLM$ in~\cref{eqn:MfromData} is unitarily equivalent to a
real-valued matrix $\widetilde{\bbLM} \defeq (\BI_{N/2} \otimes
\BJ_{p}^{\herm}) \bbLM (\BI_{N/2}\otimes \BJ_{m}) \in \R^{p K \times m J}$,
where $\otimes$ denotes the matrix Kronecker product~\cite[Chap.~12.3]{GolV13}.
Similarly, the matrices $\bbLK$, $\bbBu$ and $\bbCp$
in~\cref{eqn:KfromData,eqn:BCpCvfromData} are unitarily equivalent to
real-valued matrices as well.
Namely, one can verify that
\begin{equation*}
  \widetilde{\bbLK} \defeq (\BI_{N/2} \otimes \BJ_{p}^{\herm}) \bbLK
    (\BI_{N/2} \otimes \BJ_{m}),\quad
  \widetilde{\mathbb{B}}_{\inp} \defeq (\BI_{N/2} \otimes \BJ_{p}^{\herm}\big)
    \bbBu \quad\mbox{and}\quad
  \widetilde{\mathbb{C}}_{\pos} = \bbCp (\BI_{N/2} \otimes \BJ_{m})
\end{equation*}
are all real valued by construction.
Therefore, we may use these unitary transformations to obtain real-valued data
matrices in \Cref{alg:soQuadBT} and, consequently, a real-valued realizations
of the reduced-order models.
Lastly, we mention that this realification procedure can also be applied to
the second-order Hermite Loewner matrices derived in \Cref{thm:sobtFromData_Hermite}. 
If the underlying system is real-valued and the data satisfy the hypotheses 
of~\Cref{thm:stability}, 
then this allows one to obtain a real-valued realization of an asymptotically stable system.
In this setting, it will be the case that $\zeta_{i+1} = s_{i+1} = - s_{i} = \vartheta_{i}$
and $\zeta_{i} = s_{i} = - s_{i+1} = \vartheta_{i+1}$
due to~\cref{eqn:symmNodes}, and so evaluations of the
derivative of $\BG$ will appear in the sub- and super-diagonals of $\bbLM$ and $\bbLK$.


\subsection{Optimal choice of damping parameters}%
\label{ss:optDamping}

Thus far, we have assumed that the coefficient functions $f$ and $g$ in
the generalized proportional damping model \cref{eqn:dampingModel} are given and
can be freely evaluated.
However, in a truly data-driven regime, these would be unknown. 
We discuss how these coefficients can be inferred from data in the simplified
cases of Rayleigh damping~\cref{eqn:Rayleigh} and
structural damping~\cref{eqn:structural};
this corresponds to inferring the parameters $\alpha, \beta \geq 0$ or
$\eta \geq 0$, respectively.
In either case, we assume the data $\Bg_{k} = \BG(s_{k})$ to be prescribed at a
fixed number of points $s_{k} \in \C$, for $k = 1, \ldots, N$; we do not assume
this data to have any particular structure other than that it is generated by
the true underlying system's transfer function $\BG$.
A similar idea has been employed in the previously mentioned second-order
Loewner framework~\cite{PonGB22}, as well as in the data-driven modeling of
time-delay systems~\cite{SchU16}.

We estimate the unknown parameters by minimizing a least squares fit of the given data over some realistic range(s) of parameter values, which are typically available in practice.
For the Rayleigh-damped case, the optimization problem we aim to solve is 
\begin{align}
\begin{split}
\label{eqn:rayleighOpt}
  (\alphaStar, \betaStar) & = \argmin\limits_{(\alpha, \beta) \in \intAlpha \times \intBeta}
    \sum\limits_{k = 1}^{N} \objRayleigh(\alpha, \beta), \quad\text{with}\\
  \objRayleigh(\alpha, \beta) & \defeq \left\lVert \Bg_{k} - \left(\BCpr + s_k \BCvr\right)\left((s_k^{2} + s_k \alpha) \BMr + (1 + s_k \beta) \BKr\right)^{-1} \BBur \right\rVert_{\frob}^{2}
\end{split}
\end{align}
where $\intAlpha \defeq [\alphaMinus,\alphaPlus]$ and $\intBeta \defeq [\betaMinus,\betaPlus]$.
For the structurally damped case, we have
\begin{align} 
\begin{split}
\label{eqn:structOpt}
  \etaStar &= \argmin\limits_{\eta \in \intEta} \sum\limits_{k = 1}^{N}
    \objStruct(\eta), \quad\text{with}\\
    \objStruct(\eta) &\defeq \left\lVert \Bg_{k} -\left(\BCpr + s_k \BCvr\right)\left(s_k^{2} \BMr + (1 + \i \eta ) \BKr\right)^{-1} \BBur \right\rVert_{\frob}^{2}
\end{split}
\end{align}
where $\intEta \defeq [\etaMinus,\etaPlus]$.
Any suitable off-the-shelf optimization algorithm can be used to
minimize the cost function~\cref{eqn:rayleighOpt} over the possible values of
$\alpha \in \intAlpha$ and $\beta \in \intBeta$ simultaneously if the system is Rayleigh damped,
or, in the case of structural damping, to minimize~\cref{eqn:structOpt} over
$\eta \in \intEta$.
Even though \soQuadbt{} constructs reduced models with $\BMr = \BI_{r}$, this
is not enforced here to make the presentation as general as possible. 
Thus, the minimization procedure we describe can be applied to any data-driven
reduced model~\cref{eqn:soSysRed} with unknown damping parameters; e.g.,
it can be applied to those computed by the second-order Loewner
framework~\cite{PonGB22}.

We emphasize that the reduced-order matrices may depend explicitly on the parameters $\alpha$, $\beta$, or $\eta$ as in the case of \soQuadbt{} (\Cref{alg:soQuadBT}).
However, we omit this dependence in our notation to simplify the presentation.
In this case, the reduced-order matrices need to be recomputed for every evaluation of the cost function in the minimization procedure.
Once the optimal coefficients $\alphaStar \in \intAlpha$ and $\betaStar \in \intBeta$ or $\etaStar \in \intEta$ are computed by minimizing~\cref{eqn:rayleighOpt} or~\cref{eqn:structOpt}, respectively, a final reduced model is computed by \Cref{alg:soQuadBT} using these values.

If the damping coefficients are treated as fixed in the construction of the reduced model matrices, or if one has access to an already computed reduced model~\cref{eqn:soSysRed}, then the damping coefficients are only used to form the
reduced-order damping term from the mass and
stiffness matrices.
In this case, the objective functions~\cref{eqn:rayleighOpt,eqn:structOpt} can be minimized without reassembling the reduced model at every step. To this end, we derive gradients of the objective functions in~\cref{eqn:rayleighOpt,eqn:structOpt} when the reduced-order model matrices do not depend on the damping coefficients.
For our purposes, we employ these gradients in an optional post-processing step at the end of~\Cref{alg:soQuadBT} to form the reduced-order damping matrix.

\begin{theorem} \label{thm:grads} \allowdisplaybreaks
  Consider a reduced-order model of the form~\cref{eqn:soSysRed}.
  Let
  \begin{alignat*}{3}
    \BvarphiRayleigh & \colon \C \times [0,\infty) \times [0,\infty) \to \Cnn & \quad\text{with}\quad &&
    \BvarphiRayleigh(s, \alpha, \beta) & \defeq
    (s^{2} + s \alpha) \BMr + (1 + s\beta) \BKr, \\
    \BvarphiStruct & \colon \C \times [0,\infty) \to \Cnn & \quad\text{with}\quad &&
    \BvarphiStruct(s, \eta) & \defeq s^{2} \BMr + (1 + \i \eta ) \BKr, \\
    \CC & \colon \C \to \Cpn & \quad\text{with} \quad &&
    \CC(s) & \defeq \BCpr + s\BCvr.
  \end{alignat*}
  For the minimization problem~\cref{eqn:rayleighOpt}, the gradients with
  respect to the optimization parameters are
  \begin{subequations} \label{eqn:RayleighDeriv}
  \begin{align} \label{eqn:alphaDeriv}
    \frac{\partial}{\partial{\alpha}} \objRayleigh(\alpha, \beta) &
    = 2 \real \bigg( s_{k} \trace \bigg( \BMr \BvarphiRayleigh(s_{k}, \alpha, \beta)^{-1} \BBur\Big(\Bg_{k} 
      \\ \nonumber
    & \quad{}-{}
      \CC(s_{k})
      \BvarphiRayleigh(s_{k}, \alpha, \beta)^{-1} \BBur
      \Big)^{\herm} \CC(s_{k}) \BvarphiRayleigh(s_{k}, \alpha, \beta)^{-1}
      \bigg) \bigg), \\
    \label{eqn:betaDeriv}
    \frac{\partial}{\partial{\beta}} \objRayleigh(\alpha, \beta) &
      = 2 \real \bigg( s_{k} \trace \bigg( \BKr \BvarphiRayleigh(s_{k}, \alpha, \beta)^{-1} \BBur \Big( \Bg_{k} 
      \\ \nonumber
    & \quad{}-{}
      \CC(s_{k})
      \BvarphiRayleigh(s_{k}, \alpha, \beta)^{-1} \BBu
      \Big)^{\herm} \CC(s_{k}) \BvarphiRayleigh(s_{k}, \alpha, \beta)^{-1}
      \bigg) \bigg).
  \end{align}
  \end{subequations}
  For the minimization problem~\cref{eqn:structOpt}, the gradient is given by
  \begin{align} \label{eqn:etaDeriv}
  \frac{\partial}{\partial{\eta}} \objStruct(\eta) &
    = 2 \real \bigg( \trace \bigg( \BKr \BvarphiStruct(s_{k}, \eta)^{-1} \\ \nonumber
  & \quad{}\times{}
    \BBur \Big(\Bg_{k} - \CC(s_{k}) \BvarphiStruct(s_{k}, \eta)^{-1} \BBur
    \Big)^{\herm} \CC(s_{k}) \BvarphiStruct(s_{k}, \eta)^{-1} \bigg) \bigg).
  \end{align}
\end{theorem}
\begin{proof} \allowdisplaybreaks
  It suffices to show~\cref{eqn:betaDeriv} since the proofs
  of~\cref{eqn:RayleighDeriv} follow analogously.
  Define $\CE_{k} \defeq \Bg_{k} - \CC(s_{k}) \BvarphiRayleigh(s_{k}, \alpha,
  \beta)^{-1} \BBur \in \C^{p\times m}$ so that
  $\objRayleigh(\alpha, \beta) = \lVert \CE_{k} \rVert_{\frob}^{2} =
  \trace(\CE_{k}^{\herm} \CE_{k})$.
  Now, we consider an arbitrary perturbation $\pertBeta \in \R$, and expand
  the helper function $\BvarphiRayleigh(s_{k}, \alpha, \beta + \pertBeta)$ as
  \begin{align*}
    \BvarphiRayleigh(s_{k}, \alpha, \beta + \pertBeta) &
      = \Big( (s_{k}^{2} + s_{k} \alpha) \BMr +
      (1 + s_{k} \beta + s_{k} \pertBeta) \BKr \Big) \\
    & = \BvarphiRayleigh(s_{k}, \alpha, \beta) + s_{k} \pertBeta \BKr.
  \end{align*}
  By choosing $\pertBeta$ to be small enough such that
  $\lVert \BvarphiRayleigh(s_{k}, \alpha_{k}, \beta_{k}) s_{k} \pertBeta
  \BKr \rVert_{2} < 1$, we may expand
  $\BvarphiRayleigh(s_{k}, \alpha, \beta + \pertBeta)^{-1}$ as a Neumann series
  into
    \begin{align*}
      \BvarphiRayleigh(s_{k}, \alpha, \beta + \pertBeta)^{-1}
      & = \Big(\BvarphiRayleigh(s_{k}, \alpha, \beta) -
        (-s_{k} \pertBeta \BKr) \Big)^{-1} \\
      & = \sum\limits_{j = 0}^{\infty}
        \Big( \BvarphiRayleigh(s_{k}, \alpha, \beta)^{-1}
        (-s_{k} \pertBeta\BKr) \Big)^{j}
        \BvarphiRayleigh(s_{k}, \alpha, \beta)^{-1} \\
      & = \BvarphiRayleigh(s_{k}, \alpha, \beta)^{-1} - s_{k} \pertBeta
        \BvarphiRayleigh(s_{k}, \alpha, \beta)^{-1} \BKr
        \BvarphiRayleigh(s_{k}, \alpha, \beta)^{-1}\\
        &\quad{}+{} \mathcal{O}(\lVert \pertBeta \rVert_{\frob}^{2}).
    \end{align*}
    Using this expression, the objective function can be written as
    \begin{align*}
      & \objRayleigh(\alpha, \beta +\pertBeta) \\
      & = \lVert \CE_{k} + s_{k} \pertBeta \CC(s_{k})
        \BvarphiRayleigh(s_{k}, \alpha, \beta)^{-1} \BKr
        \BvarphiRayleigh(s_{k}, \alpha, \beta)^{-1} \BBur \rVert_{\frob}^{2}
        + \mathcal{O}(\lVert \pertBeta \rVert_{\frob}^{2}) \\
      & = \trace(\CE_{k}^{\herm} \CE_{k}) + 2 \real \Big( s_{k} \trace \Big(
        \CE_{k}^{\herm} \CC(s_{k}) \BvarphiRayleigh(s_{k}, \alpha, \beta)^{-1}
        \BKr \BvarphiRayleigh(s_{k}, \alpha, \beta)^{-1} \BBur \Big) \Big) \pertBeta\\
      &\quad{}+{} \mathcal{O}(\lVert \pertBeta \rVert_{\frob}^{2})\\
      & = \objRayleigh(\alpha, \beta) + 2 \real \Big( s_{k} \trace \Big( \BKr
        \BvarphiRayleigh(s_{k}, \alpha,\beta)^{-1} \BBur \CE_{k}^{\herm}
        \CC(s_{k}) \BvarphiRayleigh(s_{k}, \alpha, \beta)^{-1}
        \Big) \Big) \pertBeta \\
      & \quad{}+{} \mathcal{O}(\lVert \pertBeta \rVert_{\frob}^{2}),
    \end{align*}
    which proves~\cref{eqn:betaDeriv}.
\end{proof}


\section{Numerical Results}
\label{sec:numerics}

We demonstrate the proposed \soQuadbt{} method on two benchmark examples of
second-order dynamical systems from the literature.
The reported numerical experiments have been performed on a MacBook Air with
8\,GB of RAM and an Apple M2 processor running macOS Ventura version 13.4 with
MATLAB 23.2.0.2515942 (R2023b) Update 7.
The source codes, data, and results of the numerical experiments are available
at~\cite{supReiW26}.


\subsection{Experimental setup}%
\label{sec:setup}

In comparison to our proposed \soQuadbt{} approach, we consider the following
methods in the numerical experiments:
\begin{description}
  \item[\soLoewner{}] is the interpolatory Loewner framework for second-order
    systems with Rayleigh damping from~\cite{PonGB22} and highlighted in
    \Cref{remark:soLoewner};
  \item[\foQuadbt{}] is the data-driven balanced truncation for first-order
    systems from~\cite{GosGB22};
  \item[\sopvBt{}] is the intrusive position-velocity balanced truncation
    for second-order systems from~\cite{ReiS08} and discussed in
    \Cref{sec:soBt};
  \item[\Btr{}] is the intrusive unstructured balanced truncation for
    first-order systems~\cref{eqn:foSys} discussed in \Cref{sec:foBt}.
\end{description}
The intrusive \sopvBt{} is included as a point of comparison for the
data-driven aspect of \soQuadbt{}.
The intrusive and data-driven first-order methods \Btr{} and \foQuadbt{}
are included for the comparison between our proposed structured approach
and unstructured models.
The software package MORLAB version 6.0~\cite{BenW21c} is used for
computing the intrusive reduced models by \sopvBt{} and \Btr{}.

For the data-based approaches, we use an exponential trapezoidal quadrature
rule to implicitly approximate the Gramians
in~\cref{eqn:soGramians,eqn:foGramians}.
While \soLoewner{} does not typically use interpolation points derived from
quadrature rules, we fix the sampling points across all the data-based
approaches for consistency.
For simplicity, an equal number of left and right quadrature nodes is used so
that $J = K = N$, and we assume that the nodes are symmetrically distributed
according to~\cref{eqn:symmNodes}.
In each example, the quadrature rules are applied in chosen intervals along
the imaginary axis.

For the presentation of the numerical results, we use the following performance
measures.
For visual comparisons, we plot the magnitude of the second-order transfer
function~\cref{eqn:soSys} at discrete points on the positive imaginary axis.
We also show the pointwise relative approximation errors of the transfer
functions
\begin{equation} \label{eqn:relerr}
  \relerr(\i \omega_{k}) = \frac{\lVert \BG(\i \omega_{k}) -
    \BGr(\i\omega_{k}) \rVert_{2}}{\lVert \BG(\i\omega_k) \rVert_{2}},
\end{equation}
at discrete frequencies $\omega_{k} \in \Omega$ from the finite interval
$\Omega = [\omegaMin, \omegaMax] \subset [0, \infty)$.
The specific choice of $\Omega$ varies between the examples.
Additionally, we score the performance of the different methods by using local
approximations to the relative $\CH_{\infty}$ error and relative $\CH_{2}$ error via
\begin{subequations}
\label{eqn:Hrelerrs}
\begin{align} 
  \label{eqn:Hinftyrelerr}
  \relerr_{\CH_{\infty}} &= \frac{\max_{\omega_{k} \in \Omega} \lVert
    \BG(\i \omega_{k}) - \BGr(\i \omega_{k}) \rVert_{2}}%
    {\max_{\omega_{k} \in \Omega} \lVert \BG(\i \omega_{k}) \rVert_{2}},
    \quad\text{and}\\
  \label{eqn:H2relerr}
  \relerr_{\CH_{2}} &= \sqrt{\frac{\sum\limits_{\omega_{k} \in \Omega} \lVert
    \BG(\i \omega_{k}) - \BGr(\i \omega_{k}) \rVert_{\frob}^2}%
    {\sum\limits_{\omega_{k} \in \Omega} \lVert \BG(\i \omega_{k})
    \rVert_{\frob}^2}}.
\end{align}
\end{subequations}


\subsection{Butterfly Gyroscope}%
\label{sec:butterfly}

The first example that we consider is the butterfly gyroscope from the Oberwolfach
Benchmark Collection~\cite{Bil05}, which models a vibrating
mechanical gyroscope used for inertial navigation.
The model itself is a second-order system~\cref{eqn:soSys} with
$n = 17\,361$ states, $m = 1$ input and $p = 12$ position outputs.
The mass and stiffness matrices are symmetric, and the internal damping is
described by Rayleigh damping with $f(s) = \alpha = 0$ and
$g(s) = \beta = 10^{-6}$ according to~\cref{eqn:dampingModel}.
The response behavior of interest occurs in the high-frequency ranges.
We study this benchmark to illustrate how our proposed method applies to a
system with Rayleigh damping~\cref{eqn:Rayleigh}.


\subsubsection{Known damping coefficients}%
\label{sss:knownCoeff}

For the first set of experiments with the butterfly gyroscope, we assume that
the true damping coefficients $\alpha = 0$ and $\beta=10^{-6}$ are known.
Later on, in \Cref{sss:estCoeff}, we investigate the performance of \soQuadbt{}
reduced models with estimated damping coefficients $\alphaStar$ and $\betaStar$
computed using the optimization procedure from \Cref{ss:optDamping}.

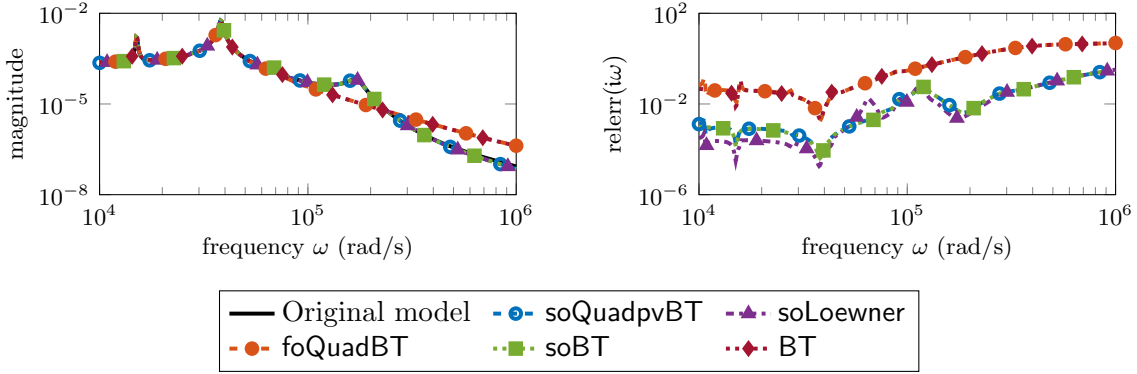
\begin{figure}[t!]
  \centering
  \begin{subfigure}[b]{.49\linewidth}
    \centering
  \tikzexternalenable%
  \tikzsetnextfilename{Butterfly_r10_mag}%
  \begin{tikzpicture}[font = \plotfontsize]
  \pgfplotstableread{graphics/data/butterfly_r10_N200_1e4to1e6_mag.dat}\tableINPUT
  
  \begin{loglogaxis}[%
    width  = .725\textwidth,
    height = .11\textheight,
    scale only axis,
    xmin = 1e4,
    xmax = 1e6,
    ymin = 1e-8,
    ymax = 1e-2,
    xminorticks = true,
    yminorticks = true,
    xlabel = {frequency $\omega$ (rad/s)},
    ylabel = {magnitude},
    ylabel style   = {yshift = -.3em},
    xlabel style   = {yshift = .3em},
    scaled x ticks = false,
    x tick label style = {/pgf/number format/1000 sep={\,}},
    y tick label style = {/pgf/number format/1000 sep={\,}},
    cycle list name    = plotlist
  ]
  
    \foreach \y in {1, 2, ..., 6}{
      \addplot+ table[x index = 0, y index = \y] {\tableINPUT};
    }
  \end{loglogaxis}
\end{tikzpicture}%
  \tikzexternaldisable%

  \end{subfigure}%
  \hfill%
  \begin{subfigure}[b]{.49\linewidth}
    \centering
  \tikzexternalenable%
  \tikzsetnextfilename{Butterfly_r10_error}%
  \begin{tikzpicture}[font = \plotfontsize]
  \pgfplotstableread{graphics/data/butterfly_r10_N200_1e4to1e6_error.dat}\tableINPUT
  
  \begin{loglogaxis}[%
    width  = .725\textwidth,
    height = .11\textheight,
    scale only axis,
    xmin = 1e4,
    xmax = 1e6,
    ymin = 1e-6,
    ymax = 1e2,
    xminorticks = true,
    yminorticks = true,
    xlabel = {frequency $\omega$ (rad/s)},
    ylabel = {$\relerr(\i\omega)$},
    ylabel style   = {yshift = -.3em},
    xlabel style   = {yshift = .3em},
    scaled x ticks = false,
    x tick label style = {/pgf/number format/1000 sep={\,}},
    y tick label style = {/pgf/number format/1000 sep={\,}},
    cycle list name    = plotlist
  ]

  \pgfplotsset{cycle list shift = 1}
  
    \foreach \y in {1, 2, ..., 5}{
      \addplot+ table[x index = 0, y index = \y] {\tableINPUT};
    }
  \end{loglogaxis}
\end{tikzpicture}%
  \tikzexternaldisable%

  \end{subfigure}

  \vspace{.5\baselineskip}
  \tikzexternalenable%
  \tikzsetnextfilename{legend}%
  \begin{tikzpicture}
  \begin{axis}[%
    hide axis,
    width  = 1mm,
    height = 1mm,
    scale only axis,
    xmin = 0,
    xmax = 1,
    ymin = 0,
    ymax = 1,
    legend columns = 3, 
    legend style   = {
      at     = {(0,0)},
      anchor = center,
      /tikz/every even column/.append style = {column sep = 0.2cm}},
    legend cell align  = {left},
    clip mode          = individual,
    cycle list name    = plotlist]

    \foreach \y in {1, 2, ..., 6}{
      \addplot+ coordinates{ (0, 0) };
    }
    
    \addlegendentry{Original model}
    \addlegendentry{\soQuadbt{}}
    \addlegendentry{\soLoewner{}}
    \addlegendentry{\foQuadbt{}}
    \addlegendentry{\soBt{}}
    \addlegendentry{\Btr{}}
  \end{axis}
\end{tikzpicture}%
  \tikzexternaldisable%

  \caption{Frequency response and pointwise relative errors for $r=10$
    reduced-order models of the butterfly gyroscope benchmark with known
    damping parameters.
    The data-driven balancing methods and their intrusive counterparts are
    indistinguishable.}
  \label{fig:ButterflyNumerics}
\end{figure}

\begin{table}[t!]
  \centering
  \begin{tabular}{llllll}
    \hline\noalign{\smallskip}
      & \multicolumn{1}{c}{\soQuadbt}
      & \multicolumn{1}{c}{\soLoewner}
      & \multicolumn{1}{c}{\foQuadbt}
      & \multicolumn{1}{c}{\sopvBt}
      & \multicolumn{1}{c}{\Btr}\\
      \noalign{\smallskip}\hline\noalign{\smallskip}
      $\relerr_{\CH_\infty}$
      & $4.4215\texttt{e-}4$
      & $4.5263\texttt{e-}4$
      & $1.0498\texttt{e-}2$
      & $\boldsymbol{4.4202\texttt{e-}4}$
      & $1.0505\texttt{e-}2$\\
      $\relerr_{\CH_2}$
      & $8.0512\texttt{e-}4$
      & $8.3923\texttt{e-}4$
      & $2.5685\texttt{e-}2$
      & $\boldsymbol{8.0166\texttt{e-}4}$
      & $2.5297\texttt{e-}2$\\
      \noalign{\smallskip}\hline\noalign{\smallskip}
  \end{tabular}
  \caption{Relative $\CH_{\infty}$ errors~\cref{eqn:Hinftyrelerr} and
    $\CH_2$ errors~\cref{eqn:H2relerr} for the $r = 10$ reduced-order models of the butterfly gyroscope with known damping coefficients.
    The smallest error is highlighted in \textbf{boldface}.}
  \label{tab:ButterflyRelErrors}
\end{table}

Surrogate models of order $r=10$ are computed using \soQuadbt{}, \soLoewner{},
\foQuadbt{}, \sopvBt{} and \Btr{}. 
For the data-based approaches, we use $N = 200$ left and right quadrature nodes.
These are chosen to be interwoven logarithmically spaced points in the interval
$-\i[10^{4}, 10^{6}] \cup \i[10^{4}, 10^{6}]$ and symmetrically distributed
according to~\cref{eqn:symmNodes}.
We use the realification process described in \Cref{sec:realmodel} for the
\soQuadbt{} and \soLoewner{} reduced models; for \foQuadbt{}, a similar
procedure from~\cite[Sec.~4.1]{GosGB22} is applied.
The frequency response and pointwise relative errors~\cref{eqn:relerr} of the
reduced-order models are presented in \Cref{fig:ButterflyNumerics}.
We see that the data-driven approaches \soQuadbt{} and \foQuadbt{} are visibly
indistinguishable from their intrusive counterparts \soBt{} and \Btr{}.
Also, we observe that each of the approaches that preserve the underlying
second-order structure, namely \soQuadbt{}, \soLoewner{} and \sopvBt{}, 
provide very satisfactory approximations and are able to capture the response
peaks of the full-order transfer function.
By comparison, the unstructured methods \foQuadbt{} and \Btr{} perform roughly
two orders of magnitude worse throughout the entire frequency range, and
entirely miss the response between $10^{5}$ and $10^{6}$\,rad/s.
This results from the fact that the first-order approximations having the same
dimension as the second-order approximations cannot provide the same level of
accuracy, regardless of whether that approximation is data-driven or intrusive.
\Cref{tab:ButterflyRelErrors} shows the relative error
measures~\cref{eqn:Hinftyrelerr,eqn:H2relerr} for the different methods.
The structure-preserving methods clearly outperform the unstructured ones, and
the \soQuadbt{} and \soBt{} reduced models produce very similar
error values.


\subsubsection{Estimating damping coefficients}%
\label{sss:estCoeff}

Next, we employ the minimization procedure from \Cref{ss:optDamping} to infer the
Rayleigh damping coefficients from data, and use these estimated coefficients
$\alphaStar$ and $\betaStar$ in the construction of the reduced-order models.
For this set of experiments, we only consider \soQuadbt{} to highlight the
effects of the estimated coefficients on the reduced model error.
The data $\Bg_{k}$ used to infer the damping coefficients are the evaluations
of the transfer function $\BG(\i \vartheta_{k})$ and $\BG(\i\zeta_{k})$ at the
left and right quadrature points.
For solving the minimization problem~\cref{eqn:rayleighOpt}, we use the optimizer \texttt{particleswarm} provided by MATLAB's global optimization toolbox.
Once suitable minima $\alphaStar$ and $\betaStar$ are determined and a reduced model is computed by~\Cref{alg:soQuadBT} using these parameters,
we employ the post-processing step described in~\Cref{ss:optDamping} using the gradients~\cref{eqn:RayleighDeriv} from \Cref{thm:grads} to form the reduced-order damping term. For this post-processing step,
the built-in MATLAB optimizer \texttt{fminunc} is used;
the internal minimization algorithm is the BFGS method with step and optimality
tolerances both set to $10^{-10}$.

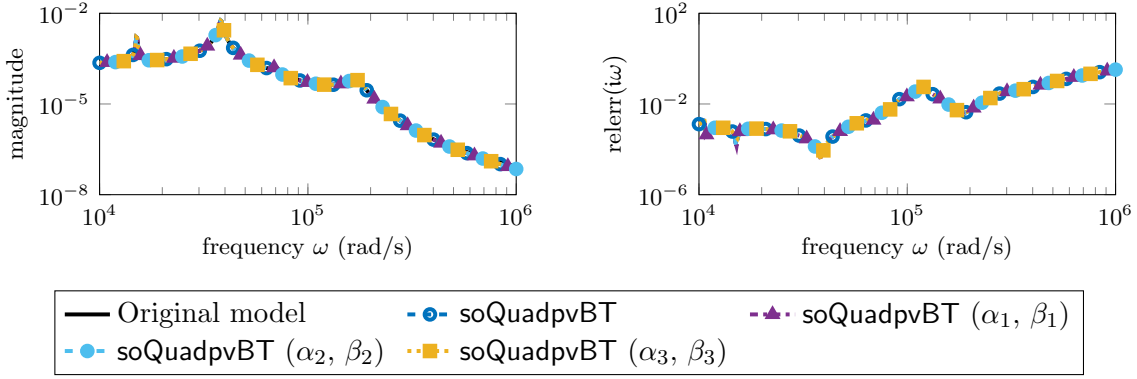
\begin{figure}[t!]
  \centering
  \begin{subfigure}[b]{.49\linewidth}
    \centering
  \tikzexternalenable%
  \tikzsetnextfilename{Butterfly_dampingOpt_r10_mag}%
  \begin{tikzpicture}[font = \plotfontsize]
  \pgfplotstableread{graphics/data/butterfly_dampingOpt_r10_N200_1e4to1e6_mag.dat}\tableINPUT
  
  \begin{loglogaxis}[%
    width  = .725\textwidth,
    height = .11\textheight,
    scale only axis,
    xmin = 1e4,
    xmax = 1e6,
    ymin = 1e-8,
    ymax = 1e-2,
    xminorticks = true,
    yminorticks = true,
    xlabel = {frequency $\omega$ (rad/s)},
    ylabel = {magnitude},
    ylabel style   = {yshift = -.3em},
    xlabel style   = {yshift = .3em},
    scaled x ticks = false,
    x tick label style = {/pgf/number format/1000 sep={\,}},
    y tick label style = {/pgf/number format/1000 sep={\,}},
    cycle list name    = damping_plotlist
  ]
  
    \foreach \y in {1, 2, ..., 4}{
      \addplot+ table[x index = 0, y index = \y] {\tableINPUT};
    }
  \end{loglogaxis}
\end{tikzpicture}%
  \tikzexternaldisable%

  \end{subfigure}%
  \hfill%
  \begin{subfigure}[b]{.49\linewidth}
    \centering
  \tikzexternalenable%
  \tikzsetnextfilename{Butterfly_dampingOpt_r10_error}%
  \begin{tikzpicture}[font = \plotfontsize]
  \pgfplotstableread{graphics/data/butterfly_dampingOpt_r10_N200_1e4to1e6_error.dat}\tableINPUT
  
  \begin{loglogaxis}[%
    width  = .725\textwidth,
    height = .11\textheight,
    scale only axis,
    xmin = 1e4,
    xmax = 1e6,
    ymin = 1e-6,
    ymax = 1e2,
    xminorticks = true,
    yminorticks = true,
    xlabel = {frequency $\omega$ (rad/s)},
    ylabel = {$\relerr(\i\omega)$},
    ylabel style   = {yshift = -.3em},
    xlabel style   = {yshift = .3em},
    scaled x ticks = false,
    x tick label style = {/pgf/number format/1000 sep={\,}},
    y tick label style = {/pgf/number format/1000 sep={\,}},
    cycle list name    = damping_plotlist
  ]

  \pgfplotsset{cycle list shift = 1}
  
    \foreach \y in {1, 2, 3}{
      \addplot+ table[x index = 0, y index = \y] {\tableINPUT};
    }
  \end{loglogaxis}
\end{tikzpicture}%
  \tikzexternaldisable%

  \end{subfigure}

  \vspace{.5\baselineskip}
  \tikzexternalenable%
  \tikzsetnextfilename{damping_legend}%
  \begin{tikzpicture}
  \begin{axis}[%
    hide axis,
    width  = 1mm,
    height = 1mm,
    scale only axis,
    xmin = 0,
    xmax = 1,
    ymin = 0,
    ymax = 1,
    legend columns = 2, 
    legend style   = {
      at     = {(0,0)},
      anchor = center,
      /tikz/every even column/.append style = {column sep = 0.2cm}},
    legend cell align  = {left},
    clip mode          = individual,
    cycle list name    = damping_plotlist]

    \foreach \y in {1, 2, ..., 5}{
      \addplot+ coordinates{ (0, 0) };
    }
    
    \addlegendentry{Original model}
    \addlegendentry{\soQuadbt{}}
    \addlegendentry{\soQuadbt{} $\intAlpha^{1}\times\intBeta^{1}$}
    \addlegendentry{\soQuadbt{} $\intAlpha^{2}\times \intBeta^{2}$}
  \end{axis}
\end{tikzpicture}%
  \tikzexternaldisable%

  \caption{Frequency response and pointwise relative errors
    for the $r = 10$ reduced-order \soQuadbt{} models of the butterfly gyroscope
    using the true and estimated damping coefficients.}
  \label{fig:ButterflyDampingNumerics}
\end{figure}

\begin{table}[t!]
 \centering
  \resizebox{\linewidth}{!}{
  \begin{tabular}{lcc}
    \hline\noalign{\smallskip} 
      & \multicolumn{1}{c}{$\intAlpha^{1}\times\intBeta^{1}=[0, 10^{-8}]\times[10^{-7}, 10^{-5}]$}
      & \multicolumn{1}{c}{$\intAlpha^{2}\times\intBeta^{2}=[0, 10^{-8}]\times[10^{-8}, 10^{-4}]$}\\
      \noalign{\smallskip}\hline\noalign{\smallskip}
      Estimated $\alphaStar$
      & $8.4361\texttt{e-}9$
      & $1.0000\texttt{e-}8$\\
      Estimated $\betaStar$
      & $9.8017\texttt{e-}9$
      & $3.8284\texttt{e-}5$\\
      \noalign{\smallskip}\hline\noalign{\smallskip}
  \end{tabular}}
  \caption{Estimated Rayleigh damping coefficients $(\alphaStar, \betaStar)$
    for the butterfly gyroscope example over different intervals $\intAlpha^{i}$ and $\intBeta^{i}$.}
  \label{tab:optDampingSolns}
\end{table} 

The minimization is applied over two different pairs of intervals, namely $\intAlpha^{1} \times \intBeta^{1} = [0, 10^{-8}] \times [10^{-7}, 10^{-5}]$ and $\intAlpha^{2} \times \intBeta^{2} = [0, 10^{-8}] \times [10^{-8}, 10^{-4}]$, and the resulting local optima are shown in~\Cref{tab:optDampingSolns}.
Three different \soQuadbt{} reduced-order models are computed based
on the true damping parameters $(\alpha, \beta) = (0, 10^{-6})$ and the two
sets of inferred coefficients from \Cref{tab:optDampingSolns}.
The frequency response and pointwise relative errors of these reduced models
are all shown in \Cref{fig:ButterflyDampingNumerics}.
The pointwise relative errors in \Cref{fig:ButterflyDampingNumerics} reveal
that the \soQuadbt{} reduced model that uses the estimated $\alphaStar$ and
$\betaStar$ corresponding to $\intAlpha^{1}$ and $\intBeta^{1}$ are nearly indistinguishable from the reduced model that
incorporates the true $\alpha = 0$ and $\beta = 10^{-6}$.
On the other hand, the reduced model that uses the coefficients estimated by optimizing over $\intAlpha^{2}$ and $\intBeta^{2}$ performs a few orders of magnitude worse. This is because the particles get stuck in a stationary point of the objective function~\cref{eqn:rayleighOpt}.
This shows that we can use \soQuadbt{} to obtain suitable second-order reduced
models as well as suitable damping coefficients using only the given data if we have a sufficiently good initial search interval.


\subsubsection{Hermitian approach and stability preservation}%
\label{sss:stabPres}

Finally, we test the construction formulas from \Cref{thm:sobtFromData_Hermite}
to validate the stability result in \Cref{thm:stability}.
To make the system symmetric, we modify the input and position-output matrices
to be $\BBu = \BCp = \tfrac{1}{12} \Bone_{n}$, where $\Bone_{n} \in \Rn$ is the
vector of all ones of length $n$.
For data generation, we use $N = 200$ logarithmically spaced points $s_{i}$ in
the interval $-\i[10^{4}, 10^{6}] \cup \i[10^{4}, 10^{6}]$ and take the
left and right quadrature nodes to be such that $\vartheta_{k} = -s_{k}$ and
$\zeta_{j} = s_{j}$, which are symmetrically distributed according
to~\cref{eqn:symmNodes}, and the weights satisfying
$\xi_{i} = \varphi_{i} = \varrho_{i}$.
Specifically, this organizational scheme implies that $s_{i} = -s_{i+1}$, for
$i = 1, 3, \ldots, N - 1$ and so the super and sub-diagonal entries of the
matrices $\bbLM$ and $\bbLK$ require derivatives of $\BG$ as
in~\cref{eqn:dataFormulas_Hermite}.

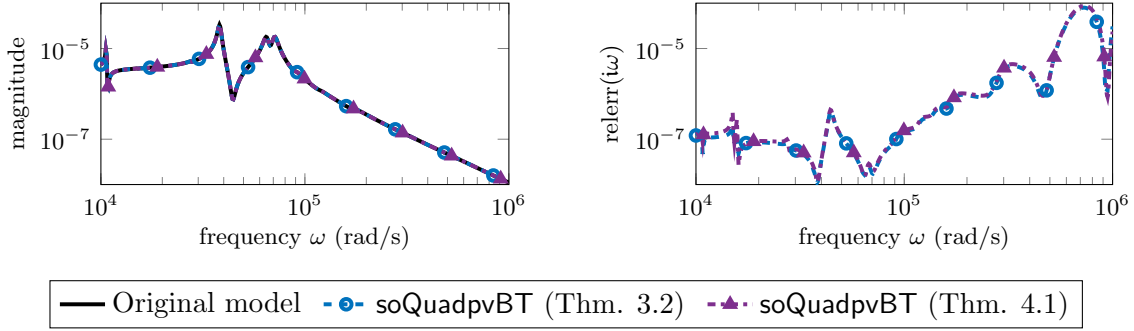
\begin{figure}[t!]
  \centering
  \begin{subfigure}[b]{.49\linewidth}
    \centering
  \tikzexternalenable%
  \tikzsetnextfilename{Butterfly_Hermite_r20_mag}%
  \begin{tikzpicture}[font = \plotfontsize]
  \pgfplotstableread{graphics/data/butterfly_Hermite_r20_N200_1e4to1e6_mag.dat}\tableINPUT
  
  \begin{loglogaxis}[%
    width  = .71\textwidth,
    height = .11\textheight,
    scale only axis,
    xmin = 1e4,
    xmax = 1e6,
    ymin = 1e-8,
    ymax = 1e-4,
    xminorticks = true,
    yminorticks = true,
    xlabel = {frequency $\omega$ (rad/s)},
    ylabel = {magnitude},
    ylabel style   = {yshift = -.3em},
    xlabel style   = {yshift = .3em},
    scaled x ticks = false,
    x tick label style = {/pgf/number format/1000 sep={\,}},
    y tick label style = {/pgf/number format/1000 sep={\,}},
    cycle list name    = plotlist
  ]
  
    \foreach \y in {1, 2, 3}{
      \addplot+ table[x index = 0, y index = \y] {\tableINPUT};
    }
  \end{loglogaxis}
\end{tikzpicture}%
  \tikzexternaldisable%

  \end{subfigure}%
  \hfill%
  \begin{subfigure}[b]{.49\linewidth}
    \centering
  \tikzexternalenable%
  \tikzsetnextfilename{Butterfly_Hermite_r20_error}%
  \begin{tikzpicture}[font = \plotfontsize]
  \pgfplotstableread{graphics/data/butterfly_Hermite_r20_N200_1e4to1e6_error.dat}\tableINPUT
  
  \begin{loglogaxis}[%
    width  = .725\textwidth,
    height = .11\textheight,
    scale only axis,
    xmin = 1e4,
    xmax = 1e6,
    ymin = 1e-8,
    ymax = 1e-4,
    xminorticks = true,
    yminorticks = true,
    xlabel = {frequency $\omega$ (rad/s)},
    ylabel = {$\relerr(\i\omega)$},
    ylabel style   = {yshift = -.3em},
    xlabel style   = {yshift = .3em},
    scaled x ticks = false,
    x tick label style = {/pgf/number format/1000 sep={\,}},
    y tick label style = {/pgf/number format/1000 sep={\,}},
    cycle list name    = plotlist
  ]

  \pgfplotsset{cycle list shift = 1}
  
    \foreach \y in {1, 2}{
      \addplot+ table[x index = 0, y index = \y] {\tableINPUT};
    }
  \end{loglogaxis}
\end{tikzpicture}%
  \tikzexternaldisable%

  \end{subfigure}

  \vspace{.5\baselineskip}
  \tikzexternalenable%
  \tikzsetnextfilename{hermite_legend}%
  \begin{tikzpicture}
  \begin{axis}[%
    hide axis,
    width  = 1mm,
    height = 1mm,
    scale only axis,
    xmin = 0,
    xmax = 1,
    ymin = 0,
    ymax = 1,
    legend columns = 3, 
    legend style   = {
      at     = {(0,0)},
      anchor = center,
      /tikz/every even column/.append style = {column sep = 0.2cm}},
    legend cell align  = {left},
    clip mode          = individual,
    cycle list name    = plotlist]

    \foreach \y in {1, 2, 3}{
      \addplot+ coordinates{ (0, 0) };
    }
    
    \addlegendentry{Original model}
    \addlegendentry{\soQuadbt{} (Thm. 3.2)}
    \addlegendentry{\soQuadbt{} (Thm. 4.1)}
  \end{axis}
\end{tikzpicture}%
  \tikzexternaldisable%

  \caption{Frequency response and pointwise relative errors for $r = 20$
    \soQuadbt{} reduced-order models of the butterfly gyroscope using the data
    matrices from \Cref{thm:sobtFromData} and \Cref{thm:sobtFromData_Hermite}.
    The stability-preserving model performs as well as the standard approach in
    terms of accuracy.}
  \label{fig:ButterflyHermiteNumerics}
\end{figure}

\begin{table}[t!]
  \centering
  
  \resizebox{\linewidth}{!}{
  \begin{tabular}{lcccccc}
    \hline\noalign{\smallskip} 
      & \multicolumn{1}{c}{$r = 10$}
      & \multicolumn{1}{c}{$r = 12$}
      & \multicolumn{1}{c}{$r = 14$}
      & \multicolumn{1}{c}{$r = 16$}
      & \multicolumn{1}{c}{$r = 18$}
      & \multicolumn{1}{c}{$r = 20$}\\
      \noalign{\smallskip}\hline\noalign{\smallskip}
      \soQuadbt{} (\Cref{thm:sobtFromData})
      & \checkmark
      & \checkmark
      & \checkmark
      & \checkmark
      & \checkmark
      & \checkmark\\
      \soQuadbt{} (\Cref{thm:sobtFromData_Hermite})
      & \checkmark
      & \checkmark
      & \checkmark
      & \checkmark
      & \checkmark
      & \checkmark\\
      \noalign{\smallskip}\hline\noalign{\smallskip}
  \end{tabular}}
  \caption{Stability of reduced models computed by \soQuadbt{} using
    \Cref{thm:sobtFromData,thm:sobtFromData_Hermite} for varying orders
    $r = 10, 12, \ldots, 20$.
    Check marks indicate asymptotic stability, and all computed reduced-order
    models turned out to be asymptotically stable.}
  \label{tab:stabCheck}
\end{table}

For varying reduction orders  $r = 10, 12, \ldots, 20$, we compute
reduced-order models via \soQuadbt{} using the setup just described and the
data matrices from~\Cref{thm:sobtFromData_Hermite}, and for the comparison, we
also compute reduced-order models using the setup of \Cref{sss:knownCoeff} and
the data matrices in~\Cref{thm:sobtFromData}.
In all cases, we employ the realification procedure from \Cref{sec:realmodel}
to compute real-valued surrogates.
The response and pointwise relative errors~\cref{eqn:relerr} of the
$r = 20$ reduced-order models are plotted in \Cref{fig:ButterflyHermiteNumerics}.
As illustrated by the figure, the \soQuadbt{} reduced models using either set
of data matrices provide very satisfactory approximations over the entire
frequency range, and are nearly indistinguishable from each other.
In \Cref{tab:stabCheck}, we report the stability properties of the
reduced-order models.
As expected from \Cref{thm:stability}, the \soQuadbt{} models constructed via
\Cref{thm:sobtFromData_Hermite} are all asymptotically stable by construction. 
Although not guaranteed, the \soQuadbt{} reduced models constructed using
\Cref{thm:sobtFromData} turned out to be asymptotically stable, as well.


\subsection{Plate with tuned vibration absorbers}%
\label{sec:plateTVA}

As a second example, we consider the model of the vibrational response of
a strutted plate from~\cite[Sec.~4.2]{AumW23}.
We refer the reader to~\cite[Sec.~4.2]{AumW23} for the material parameters and
model details.
The model is a second-order system~\cref{eqn:soSys} with
$n = 209\,100$ states, $m = 1$ input and $p = 1$ position output, and
it uses the structural (hysteretic) damping formulation with
$g(s) = (\i \eta) / s$ for $\eta = 10^{-3}$.
The tuned vibration absorbers (TVAs) connected to the plate reduce
the vibrational response in the frequency region about $48$\,Hz.

In this example, we only employ the data-driven methods \soQuadbt{},
\soLoewner{} and \foQuadbt{} and omit the intrusive model reduction counterparts due to the lack of available solvers for the algebraic Lyapunov equations with complex-valued coefficient matrices.
Furthermore, we do not employ the realification process here because the
underlying hysteretic damping leads to the full-order model being complex
valued.
For generating the data, we choose $N = 250$ left and right quadrature nodes
that are interwoven linearly spaced points in the interval
$-2 \pi \i[0, 250] \cup 2 \pi \i[0, 250]$ and symmetrically distributed
according to~\cref{eqn:symmNodes}.
We use the aforementioned methods to compute reduced-order models of dimension
$r = 75$.
The resulting frequency responses and pointwise relative errors are
plotted in \Cref{fig:plateNumerics}.
Similar to the previous example, the structure-preserving approaches
\soQuadbt{} and \soLoewner{} outperform the unstructured approach \foQuadbt{}
for most frequencies by about five orders of magnitude.
The relative error measures reported in \Cref{tab:PlateRelErrors} also support
this claim.
We observe that all three reduced models struggle to produce a good
approximation in the region around $48$\,Hz, which is due to the TVAs that
heavily damp the frequency response about that region.

\begin{figure}[t!]
  \centering
  \begin{subfigure}[b]{.49\linewidth}
    \centering
  \tikzexternalenable%
  \tikzsetnextfilename{soplate_r75_mag}%
  \begin{tikzpicture}[font = \plotfontsize]
  \pgfplotstableread{graphics/data/plateTVA_r75_N250_0to250Hz_mag.dat}\tableINPUT
  
  \begin{axis}[%
    width  = .725\textwidth,
    height = .11\textheight,
    scale only axis,
    xmin = 0,
    xmax = 250,
    ymin = -10,
    ymax = 48,
    xminorticks = true,
    yminorticks = true,
    xlabel = {frequency (Hz)},
    ylabel = {magnitude (dB)},
    ylabel style   = {yshift = -.3em},
    xlabel style   = {yshift = .3em},
    scaled x ticks = false,
    x tick label style = {/pgf/number format/1000 sep={\,}},
    y tick label style = {/pgf/number format/1000 sep={\,}},
    cycle list name    = plotlist
  ]
  
    \foreach \y in {1, 2, ..., 4}{
      \addplot+ table[x index = 0, y index = \y] {\tableINPUT};
    }
  \end{axis}
\end{tikzpicture}%
  \tikzexternaldisable%

  \end{subfigure}%
  \hfill%
  \begin{subfigure}[b]{.49\linewidth}
    \centering
  \tikzexternalenable%
  \tikzsetnextfilename{soplate_r75_errors}%
  \begin{tikzpicture}[font = \plotfontsize]
  \pgfplotstableread{graphics/data/plateTVA_r75_N250_0to250Hz_error.dat}\tableINPUT
  
  \begin{semilogyaxis}[%
    width  = .725\textwidth,
    height = .11\textheight,
    scale only axis,
    xmin = 0,
    xmax = 250,
    ymin = 10e-12,
    ymax = 10e1,
    xminorticks = true,
    yminorticks = true,
    xlabel = {frequency (Hz)},
    ylabel = {$\relerr(\i\omega)$},
    ylabel style   = {yshift = -.3em},
    xlabel style   = {yshift = .3em},
    scaled x ticks = false,
    x tick label style = {/pgf/number format/1000 sep={\,}},
    y tick label style = {/pgf/number format/1000 sep={\,}},
    cycle list name    = plateplotlist
  ]

  \pgfplotsset{cycle list shift = 1}
  
    \foreach \y in {1, 2, 3}{
      \addplot+ table[x index = 0, y index = \y] {\tableINPUT};
    }
  \end{semilogyaxis}
\end{tikzpicture}%
  \tikzexternaldisable%

  \end{subfigure}

  \vspace{.5\baselineskip}
  \tikzexternalenable%
  \tikzsetnextfilename{plate_legend}%
  \begin{tikzpicture}
  \begin{axis}[%
    hide axis,
    width  = 1mm,
    height = 1mm,
    scale only axis,
    xmin = 0,
    xmax = 1,
    ymin = 0,
    ymax = 1,
    legend columns = 4, 
    legend style   = {
      at     = {(0,0)},
      anchor = center,
      /tikz/every even column/.append style = {column sep = 0.2cm}},
    legend cell align  = {left},
    clip mode          = individual,
    cycle list name    = plotlist]

    \foreach \y in {1, 2, ..., 4}{
      \addplot+ coordinates{ (0, 0) };
    }
    
    \addlegendentry{Original model}
    \addlegendentry{\soQuadbt{}}
    \addlegendentry{\soLoewner{}}
    \addlegendentry{\foQuadbt{}}
  \end{axis}
\end{tikzpicture}%
  \tikzexternaldisable%

  \caption{Frequency response and pointwise relative errors for $r = 75$
    reduced-order models of the plate with TVAs.
    The structured approaches clearly outperform the unstructured
    \foQuadbt{} for most frequencies by about five orders of magnitude.}
  \label{fig:plateNumerics}
\end{figure}
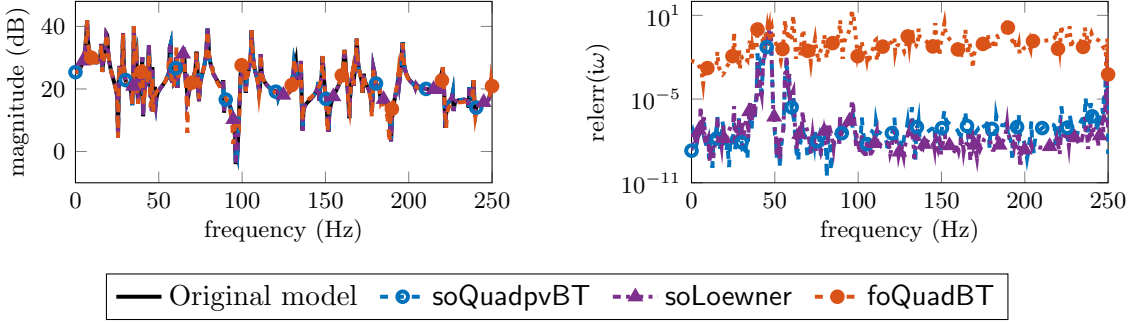

\begin{table}[t!]
  \centering
  \begin{tabular}{lllll}
    \hline\noalign{\smallskip}
      & \multicolumn{1}{c}{\soQuadbt}
      & \multicolumn{1}{c}{\soLoewner}
      & \multicolumn{1}{c}{\foQuadbt}\\
      \noalign{\smallskip}\hline\noalign{\smallskip}
      $\relerr_{\CH_\infty}$
      & $2.0148\texttt{e-}9$
      & $\boldsymbol{1.8776\texttt{e-}9}$
      & $8.7564\texttt{e-}8$\\
      $\relerr_{\CH_2}$
      & $2.7231\texttt{e-}10$
      & $\boldsymbol{2.5159\texttt{e-}10}$
      & $1.3176\texttt{e-}8$\\
      \noalign{\smallskip}\hline\noalign{\smallskip}
  \end{tabular}
  \caption{Relative $\CH_{\infty}$ errors~\cref{eqn:Hinftyrelerr} and
    $\CH_2$ errors~\cref{eqn:H2relerr} for $r = 75$ reduced-order models of
    the plate with TVAs.
    The smallest error is highlighted in \textbf{boldface}.}
  \label{tab:PlateRelErrors}
\end{table}


\subsection{Mass-spring-damper network with velocity outputs}%
\label{sec:msd}

As the final example, we consider the mass-spring damper (\MSD{}) chain
from~\cite[Ex.~2]{TruV09}. 
This is a damped linear vibration system consisting of three rows of $d$ masses
all connected to the right-hand side of a coupling mass $m_{0}$. 
The masses in the individual rows are connected by springs, and the coupling
mass $m_{0}$ is connected to the fixed left-hand side of the base.
The system order is $n = 3d + 1$.
We construct a system with $d = 300$ masses so that the full-order model
dimension is $n=901$. 
We impose Rayleigh damping on the system for the parameters
$f(s) = g(s) = 0.002$, the input $\BBu = \Bone_{n} \in \Rn$ is the vector of
all ones, and we have the outputs $\BCp = \Bzero$ and $\BCv=\Bone_{n}\in\Rn$. 
We include this example to illustrate how our method performs when velocity
outputs are incorporated.

\begin{figure}[t!]
  \centering
  \begin{subfigure}[b]{.49\linewidth}
    \centering
  \tikzexternalenable%
  \tikzsetnextfilename{MSD_Cv_r20_mag}%
  \begin{tikzpicture}[font = \plotfontsize]
  \pgfplotstableread{graphics/data/MSD_Cv_r20_N200_1e-3to1e1_mag.dat}\tableINPUT
  
  \begin{loglogaxis}[%
    width  = .725\textwidth,
    height = .11\textheight,
    scale only axis,
    xmin = 1e-3,
    xmax = 1e+1,
    ymin = 1e1,
    ymax = 1e6,
    xminorticks = true,
    yminorticks = true,
    xlabel = {frequency $\omega$ (rad/s)},
    ylabel = {magnitude},
    ylabel style   = {yshift = -.3em},
    xlabel style   = {yshift = .3em},
    scaled x ticks = false,
    x tick label style = {/pgf/number format/1000 sep={\,}},
    y tick label style = {/pgf/number format/1000 sep={\,}},
    cycle list name    = plotlist
  ]
  
    \foreach \y in {1, 2, ..., 6}{
      \addplot+ table[x index = 0, y index = \y] {\tableINPUT};
    }
  \end{loglogaxis}
\end{tikzpicture}%
  \tikzexternaldisable%

  \end{subfigure}%
  \hfill%
  \begin{subfigure}[b]{.49\linewidth}
    \centering
  \tikzexternalenable%
  \tikzsetnextfilename{MSD_Cv_r20_error}%
  \begin{tikzpicture}[font = \plotfontsize]
  \pgfplotstableread{graphics/data/MSD_Cv_r20_N200_1e-3to1e1_error.dat}\tableINPUT
  
  \begin{loglogaxis}[%
    width  = .725\textwidth,
    height = .11\textheight,
    scale only axis,
    xmin = 1e-3,
    xmax = 1e+1,
    ymin = 1e-6,
    ymax = 1e0,
    xminorticks = true,
    yminorticks = true,
    xlabel = {frequency $\omega$ (rad/s)},
    ylabel = {$\relerr(\i\omega)$},
    ylabel style   = {yshift = -.3em},
    xlabel style   = {yshift = .3em},
    scaled x ticks = false,
    x tick label style = {/pgf/number format/1000 sep={\,}},
    y tick label style = {/pgf/number format/1000 sep={\,}},
    cycle list name    = plotlist
  ]

  \pgfplotsset{cycle list shift = 1}
  
    \foreach \y in {1, 2, 3, 4, 5}{
      \addplot+ table[x index = 0, y index = \y] {\tableINPUT};
    }
  \end{loglogaxis}
\end{tikzpicture}%
  \tikzexternaldisable%

  \end{subfigure}

  \vspace{.5\baselineskip}
  \tikzexternalenable%
  \tikzsetnextfilename{legend}%
  \begin{tikzpicture}
  \begin{axis}[%
    hide axis,
    width  = 1mm,
    height = 1mm,
    scale only axis,
    xmin = 0,
    xmax = 1,
    ymin = 0,
    ymax = 1,
    legend columns = 3, 
    legend style   = {
      at     = {(0,0)},
      anchor = center,
      /tikz/every even column/.append style = {column sep = 0.2cm}},
    legend cell align  = {left},
    clip mode          = individual,
    cycle list name    = plotlist]

    \foreach \y in {1, 2, ..., 6}{
      \addplot+ coordinates{ (0, 0) };
    }
    
    \addlegendentry{Original model}
    \addlegendentry{\soQuadbt{}}
    \addlegendentry{\soLoewner{}}
    \addlegendentry{\foQuadbt{}}
    \addlegendentry{\soBt{}}
    \addlegendentry{\Btr{}}
  \end{axis}
\end{tikzpicture}%
  \tikzexternaldisable%

  \caption{Frequency response and pointwise relative errors for $r = 20$
    reduced-order models of the coupled \MSD{} network with a velocity output.
    The proposed \soQuadbt{} approach yields similar performance to its
    intrusive counterpart \soBt{}.}
  \label{fig:MSDNumerics}
\end{figure}
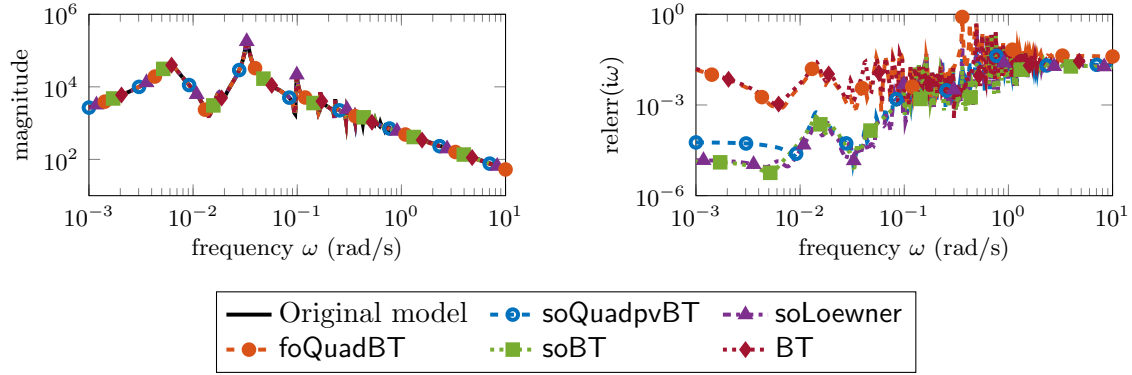

\begin{table}[t!]
  \centering
  \begin{tabular}{llllll}
    \hline\noalign{\smallskip}
      & \multicolumn{1}{c}{\soQuadbt}
      & \multicolumn{1}{c}{\soLoewner}
      & \multicolumn{1}{c}{\foQuadbt}
      & \multicolumn{1}{c}{\sopvBt}
      & \multicolumn{1}{c}{\Btr}\\
      \noalign{\smallskip}\hline\noalign{\smallskip}
      $\relerr_{\CH_\infty}$
      & $1.2550\texttt{e-}3$
      & $2.9718\texttt{e-}3$
      & $4.2038\texttt{e-}3$
      & $\boldsymbol{9.3830\texttt{e-}4}$
      & $3.6199\texttt{e-}3$\\
      $\relerr_{\CH_2}$
      & $1.0782\texttt{e-}3$
      & $1.5801\texttt{e-}3$
      & $4.2038\texttt{e-}3$
      & $\boldsymbol{8.5111\texttt{e-}4}$
      & $4.8201\texttt{e-}3$\\
      \noalign{\smallskip}\hline\noalign{\smallskip}
  \end{tabular}
  \caption{Relative $\CH_{\infty}$ errors~\cref{eqn:Hinftyrelerr} and
  $\CH_2$ errors~\cref{eqn:H2relerr} for $r = 20$ reduced-order models of the
  \MSD{} chain.
  The smallest error is highlighted in \textbf{boldface}.}
  \label{tab:MSDErrors}
\end{table}

We compute reduced-order models of size $r=20$ using
\soQuadbt{}, \soLoewner{}, \foQuadbt{}, \sopvBt{} and \Btr{}. 
For generating the data, we choose $N = 200$ left and right quadrature nodes as
interwoven logarithmically spaced points in the interval
$-\i[10^{-3}, 10^{1}] \cup \i[10^{-3}, 10^{1}]$, which are symmetrically
distributed according to~\cref{eqn:symmNodes}.
The frequency response of the full and reduced-order transfer functions as well
as their pointwise relative errors~\cref{eqn:relerr} are plotted in
\Cref{fig:MSDNumerics}.
All methods produce reasonably good approximations to the full-order system
response.
As was the case for the previous two benchmark problems, we again observe that
the structure-preserving methods \soQuadbt{}, \soLoewner{}, and \soBt{}
outperform the unstructured methods \foQuadbt{} and \Btr{}.
This performance gap is particularly evident at the lower frequencies; all four
methods exhibit roughly similar behavior from $1$\,rad/s onward. 
The relative error measures are reported in \Cref{tab:MSDErrors}, which shows
that \sopvBt{} performs the best overall.


\section{Conclusions}%
\label{sec:conclusions}

We introduced a data-driven formulation of the second-order position-velocity 
balanced truncation approach for the construction of structured 
reduced-order systems solely from transfer function data.
The approach is based on the quadrature formulation of the Gramian integrals
used in the balancing process.
We provide several modifications of the approach for practical situations
including the preservation of stability in the case of symmetric second-order
systems, the construction of real-valued systems, and the inference of optimal
damping parameters from given data.
The numerical experiments verify that the data-driven formulation of the
position-velocity balanced truncation is indistinguishable from its intrusive
counterpart for accurate enough quadrature rules, and they show the
effectiveness of the approach to construct structured reduced-order models,
which successfully outperform unstructured data-driven models of the same size.


\addcontentsline{toc}{section}{References}
\bibliographystyle{plainurl}
\bibliography{root}

\begin{thebibliography}{10}

\bibitem{AckGGetal25}
M.~S. Ackermann, I.~V. Gosea, S.~Gugercin, and S.~W.~R. Werner.
\newblock Second-order {AAA} algorithms for structured data-driven modeling.
\newblock e-print 2506.02241, arXiv, 2025.
\newblock Numerical Analysis (math.NA).
\newblock \href {https://doi.org/10.48550/arXiv.2506.02241}
  {\path{doi:10.48550/arXiv.2506.02241}}.

\bibitem{Ant05}
A.~C. Antoulas.
\newblock {\em Approximation of Large-Scale Dynamical Systems}, volume~6 of
  {\em Adv. Des. Control}.
\newblock SIAM, Philadelphia, PA, 2005.
\newblock \href {https://doi.org/10.1137/1.9780898718713}
  {\path{doi:10.1137/1.9780898718713}}.

\bibitem{AumW23}
Q.~Aumann and S.~W.~R. Werner.
\newblock Structured model order reduction for vibro-acoustic problems using
  interpolation and balancing methods.
\newblock {\em J. Sound Vib.}, 543:117363, 2023.
\newblock \href {https://doi.org/10.1016/j.jsv.2022.117363}
  {\path{doi:10.1016/j.jsv.2022.117363}}.

\bibitem{BaiMS05}
Z.~Bai, K.~Meerbergen, and Y.~Su.
\newblock {A}rnoldi methods for structure-preserving dimension reduction of
  second-order dynamical systems.
\newblock In P.~Benner, V.~Mehrmann, and D.~C. Sorensen, editors, {\em
  Dimension Reduction of Large-Scale Systems}, volume~45 of {\em Lect. Notes
  Comput. Sci. Eng.}, pages 173--189. Springer, Berlin, Heidelberg, 2005.
\newblock \href {https://doi.org/10.1007/3-540-27909-1_7}
  {\path{doi:10.1007/3-540-27909-1_7}}.

\bibitem{BaiS05a}
Z.~Bai and Y.~Su.
\newblock Dimension reduction of large-scale second-order dynamical systems via
  a second-order {A}rnoldi method.
\newblock {\em {SIAM} J. Sci. Comput.}, 26(5):1692--1709, 2005.
\newblock \href {https://doi.org/10.1137/040605552}
  {\path{doi:10.1137/040605552}}.

\bibitem{BeaB14}
C.~Beattie and P.~Benner.
\newblock {$\mathcal{H}_2$}-optimality conditions for structured dynamical
  systems.
\newblock Preprint MPIMD/14-18, Max Planck Institute for Dynamics of Complex
  Technical Systems Magdeburg, 2014.
\newblock URL: \url{https://csc.mpi-magdeburg.mpg.de/preprints/2014/18/}.

\bibitem{BeaG09}
C.~A. Beattie and S.~Gugercin.
\newblock Interpolatory projection methods for structure-preserving model
  reduction.
\newblock {\em Syst. Control Lett.}, 58(3):225--232, 2009.
\newblock \href {https://doi.org/10.1016/j.sysconle.2008.10.016}
  {\path{doi:10.1016/j.sysconle.2008.10.016}}.

\bibitem{BenB17}
P.~Benner and T.~Breiten.
\newblock Model order reduction based on system balancing.
\newblock In P.~Benner, M.~Ohlberger, A.~Cohen, and K.~Willcox, editors, {\em
  Model Reduction and Approximation: {T}heory and Algorithms}, Computational
  Science \& Engineering, pages 261--295. SIAM, Philadelphia, PA, 2017.
\newblock \href {https://doi.org/10.1137/1.9781611974829.ch6}
  {\path{doi:10.1137/1.9781611974829.ch6}}.

\bibitem{BenW21b}
P.~Benner and S.~W.~R. Werner.
\newblock Frequency- and time-limited balanced truncation for large-scale
  second-order systems.
\newblock {\em Linear Algebra Appl.}, 623:68--103, 2021.
\newblock Special issue in honor of P.~Van~Dooren, Edited by F.~Dopico,
  D.~Kressner, N.~Mastronardi, V.~Mehrmann, and R.~Vandebril.
\newblock \href {https://doi.org/10.1016/j.laa.2020.06.024}
  {\path{doi:10.1016/j.laa.2020.06.024}}.

\bibitem{BenW21c}
P.~Benner and S.~W.~R. Werner.
\newblock {MORLAB}---{T}he {M}odel {O}rder {R}eduction {LAB}oratory.
\newblock In P.~Benner, T.~Breiten, H.~Fa{\ss}bender, M.~Hinze, T.~Stykel, and
  R.~Zimmermann, editors, {\em Model Reduction of Complex Dynamical Systems},
  volume 171 of {\em International Series of Numerical Mathematics}, pages
  393--415. Birkh{\"a}user, Cham, 2021.
\newblock \href {https://doi.org/10.1007/978-3-030-72983-7_19}
  {\path{doi:10.1007/978-3-030-72983-7_19}}.

\bibitem{Bil05}
D.~Billger.
\newblock The butterfly gyro.
\newblock In P.~Benner, V.~Mehrmann, and D.~C. Sorensen, editors, {\em
  Dimension Reduction of Large-Scale Systems}, volume~45 of {\em Lect. Notes
  Comput. Sci. Eng.}, pages 349--352. Springer, Berlin, Heidelberg, 2005.
\newblock \href {https://doi.org/10.1007/3-540-27909-1_18}
  {\path{doi:10.1007/3-540-27909-1_18}}.

\bibitem{Bla18a}
F.~Blaabjerg.
\newblock {\em Control of Power Electronic Converters and Systems: {V}olume 2}.
\newblock Academic Press, London, 2018.
\newblock \href {https://doi.org/10.1016/C2017-0-04756-0}
  {\path{doi:10.1016/C2017-0-04756-0}}.

\bibitem{Bre16}
T.~Breiten.
\newblock Structure-preserving model reduction for integro-differential
  equations.
\newblock {\em {SIAM} J. Control Optim.}, 54(6):2992--3015, 2016.
\newblock \href {https://doi.org/10.1137/15M1032296}
  {\path{doi:10.1137/15M1032296}}.

\bibitem{ChaLVetal06}
Y.~Chahlaoui, D.~Lemonnier, A.~Vandendorpe, and P.~Van~Dooren.
\newblock Second-order balanced truncation.
\newblock {\em Linear Algebra Appl.}, 415(2--3):373--384, 2006.
\newblock \href {https://doi.org/10.1016/j.laa.2004.03.032}
  {\path{doi:10.1016/j.laa.2004.03.032}}.

\bibitem{DunS88}
N.~Dunford and J.~T. Schwartz.
\newblock {\em Linear Operators. Part I: General Theory}.
\newblock John Wiley \& Sons, New York, NY, 1988.

\bibitem{Enn84}
D.~F. Enns.
\newblock Model reduction with balanced realizations: {A}n error bound and a
  frequency weighted generalization.
\newblock In {\em The 23rd IEEE Conference on Decision and Control}, pages
  127--132, 1984.
\newblock \href {https://doi.org/10.1109/CDC.1984.272286}
  {\path{doi:10.1109/CDC.1984.272286}}.

\bibitem{FilPGetal23}
Y.~Filanova, I.~Pontes~Duff, P.~Goyal, and P.~Benner.
\newblock An operator inference oriented approach for linear mechanical
  systems.
\newblock {\em Mech. Syst. Signal Process.}, 200:110620, 2023.
\newblock \href {https://doi.org/10.1016/j.ymssp.2023.110620}
  {\path{doi:10.1016/j.ymssp.2023.110620}}.

\bibitem{GolV13}
G.~H. Golub and C.~F. Van~Loan.
\newblock {\em Matrix Computations}.
\newblock Johns Hopkins Studies in the Mathematical Sciences. Johns Hopkins
  University Press, Baltimore, fourth edition, 2013.

\bibitem{GosGB22}
I.~V. Gosea, S.~Gugercin, and C.~Beattie.
\newblock Data-driven balancing of linear dynamical systems.
\newblock {\em {SIAM} J. Sci. Comput.}, 44(1):A554--A582, 2022.
\newblock \href {https://doi.org/10.1137/21M1411081}
  {\path{doi:10.1137/21M1411081}}.

\bibitem{GosGW24}
I.~V. Gosea, S.~Gugercin, and S.~W.~R. Werner.
\newblock Structured barycentric forms for interpolation-based data-driven
  reduced modeling of second-order systems.
\newblock {\em Adv. Comput. Math.}, 50(2):26, 2024.
\newblock \href {https://doi.org/10.1007/s10444-024-10118-7}
  {\path{doi:10.1007/s10444-024-10118-7}}.

\bibitem{GugA04}
S.~Gugercin and A.~C. Antoulas.
\newblock A survey of model reduction by balanced truncation and some new
  results.
\newblock {\em Int. J. Control}, 77(8):748--766, 2004.
\newblock \href {https://doi.org/10.1080/00207170410001713448}
  {\path{doi:10.1080/00207170410001713448}}.

\bibitem{LauA84}
A.~J. Laub and W.~F. Arnold.
\newblock Controllability and observability criteria for multivariable linear
  second-order models.
\newblock {\em {IEEE} Trans. Autom. Control}, 29(2):163--165, 1984.
\newblock \href {https://doi.org/10.1109/TAC.1984.1103470}
  {\path{doi:10.1109/TAC.1984.1103470}}.

\bibitem{LauHPetal87}
A.~J. Laub, M.~T. Heath, C.~C. Paige, and R.~C. Ward.
\newblock Computation of system balancing transformations and other
  applications of simultaneous diagonalization algorithms.
\newblock {\em {IEEE} Trans. Autom. Control}, 32(2):115--122, 1987.
\newblock \href {https://doi.org/10.1109/TAC.1987.1104549}
  {\path{doi:10.1109/TAC.1987.1104549}}.

\bibitem{MeyS96}
D.~G. Meyer and S.~Srinivasan.
\newblock Balancing and model reduction for second-order form linear systems.
\newblock {\em {IEEE} Trans. Autom. Control}, 41(11):1632--1644, 1996.
\newblock \href {https://doi.org/10.1109/9.544000}
  {\path{doi:10.1109/9.544000}}.

\bibitem{Moo81}
B.~C. Moore.
\newblock Principal component analysis in linear systems: controllability,
  observability, and model reduction.
\newblock {\em {IEEE} Trans. Autom. Control}, AC--26(1):17--32, 1981.
\newblock \href {https://doi.org/10.1109/TAC.1981.1102568}
  {\path{doi:10.1109/TAC.1981.1102568}}.

\bibitem{MulR76}
C.~T. Mullis and R.~A. Roberts.
\newblock Synthesis of minimum roundoff noise fixed point digital filters.
\newblock {\em {IEEE} Trans. Circuits Syst.}, 23(9):551--562, 1976.
\newblock \href {https://doi.org/10.1109/TCS.1976.1084254}
  {\path{doi:10.1109/TCS.1976.1084254}}.

\bibitem{PasA08}
B.~Pascual and S.~Adhikari.
\newblock Dynamic response of structures with frequency dependent damping
  models.
\newblock In {\em 49th {AIAA}/{ASME}/{ASCE}/{AHS}/{ASC} Structures, Structural
  Dynamics, and Materials Conference}, page 2189, 2008.
\newblock \href {https://doi.org/10.2514/6.2008-2189}
  {\path{doi:10.2514/6.2008-2189}}.

\bibitem{PerS82}
L.~Pernebo and L.~M. Silverman.
\newblock Model reduction via balanced state space representations.
\newblock {\em {IEEE} Trans. Autom. Control}, 27(2):382--387, 1982.
\newblock \href {https://doi.org/10.1109/TAC.1982.1102945}
  {\path{doi:10.1109/TAC.1982.1102945}}.

\bibitem{PonGB22}
I.~Pontes~Duff, P.~Goyal, and P.~Benner.
\newblock Data-driven identification of {R}ayleigh-damped second-order systems.
\newblock In C.~Beattie, P.~Benner, M.~Embree, S.~Gugercin, and S.~Lefteriu,
  editors, {\em Realization and Model Reduction of Dynamical Systems}, pages
  255--272. Springer, Cham, 2022.
\newblock \href {https://doi.org/10.1007/978-3-030-95157-3_14}
  {\path{doi:10.1007/978-3-030-95157-3_14}}.

\bibitem{PrzPB24}
J.~Przybilla, I.~Pontes~Duff, and P.~Benner.
\newblock Model reduction for second-order systems with inhomogeneous initial
  conditions.
\newblock {\em Syst. Control Lett.}, 183:105671, 2024.
\newblock \href {https://doi.org/10.1016/j.sysconle.2023.105671}
  {\path{doi:10.1016/j.sysconle.2023.105671}}.

\bibitem{ReiS08}
T.~Reis and T.~Stykel.
\newblock Balanced truncation model reduction of second-order systems.
\newblock {\em Math. Comput. Model. Dyn. Syst.}, 14(5):391--406, 2008.
\newblock \href {https://doi.org/10.1080/13873950701844170}
  {\path{doi:10.1080/13873950701844170}}.

\bibitem{supReiW26}
S.~Reiter and S.~W.~R. Werner.
\newblock Code, data and results for numerical experiments in ``{D}ata-driven
  balanced truncation for second-order systems with generalized proportional
  damping'' (version 1.1), January 2026.
\newblock \href {https://doi.org/10.5281/zenodo.18148591}
  {\path{doi:10.5281/zenodo.18148591}}.

\bibitem{Rei25}
S.~J. Reiter.
\newblock {\em Dimension Reduction in Structured Dynamical Systems:
  {O}ptimal-{$\mathcal{H}_{2}$} Approximation, Data-Driven Balancing, and
  Real-Time Monitoring}.
\newblock {D}issertation, Virginia Polytechnic Institute and State University,
  Blacksburg, Virginia, USA, 2025.
\newblock URL: \url{https://hdl.handle.net/10919/134223}.

\bibitem{SaaSW19}
J.~Saak, D.~Siebelts, and S.~W.~R. Werner.
\newblock A comparison of second-order model order reduction methods for an
  artificial fishtail.
\newblock {\em at-Auto\-mati\-sie\-rungs\-tech\-nik}, 67(8):648--667, 2019.
\newblock \href {https://doi.org/10.1515/auto-2019-0027}
  {\path{doi:10.1515/auto-2019-0027}}.

\bibitem{SchU16}
P.~Schulze and B.~Unger.
\newblock Data-driven interpolation of dynamical systems with delay.
\newblock {\em Syst. Control Lett.}, 97:125--131, 2016.
\newblock \href {https://doi.org/10.1016/j.sysconle.2016.09.007}
  {\path{doi:10.1016/j.sysconle.2016.09.007}}.

\bibitem{SchUBetal18}
P.~Schulze, B.~Unger, C.~Beattie, and S.~Gugercin.
\newblock Data-driven structured realization.
\newblock {\em Linear Algebra Appl.}, 537:250--286, 2018.
\newblock \href {https://doi.org/10.1016/j.laa.2017.09.030}
  {\path{doi:10.1016/j.laa.2017.09.030}}.

\bibitem{ShaNTetal24}
H.~Sharma, D.~A. Najera-Flores, M.~D. Todd, and B.~Kramer.
\newblock {L}agrangian operator inference enhanced with structure-preserving
  machine learning for nonintrusive model reduction of mechanical systems.
\newblock {\em Comput. Methods Appl. Mech. Eng.}, 423:116865, 2024.
\newblock \href {https://doi.org/10.1016/j.cma.2024.116865}
  {\path{doi:10.1016/j.cma.2024.116865}}.

\bibitem{ShaWK22}
H.~Sharma, Z.~Wang, and B.~Kramer.
\newblock {H}amiltonian operator inference: {P}hysics-preserving learning of
  reduced-order models for canonical {H}amiltonian systems.
\newblock {\em Phys. D: Nonlinear Phenom.}, 431:133122, 2022.
\newblock \href {https://doi.org/10.1016/j.physd.2021.133122}
  {\path{doi:10.1016/j.physd.2021.133122}}.

\bibitem{TomP87}
M.~S. Tombs and I.~Postlethwaite.
\newblock Truncated balanced realization of a stable non-minimal state-space
  system.
\newblock {\em Int. J. Control}, 46(4):1319--1330, 1987.
\newblock \href {https://doi.org/10.1080/00207178708933971}
  {\path{doi:10.1080/00207178708933971}}.

\bibitem{TruV09}
N.~Truhar and K.~Veseli{\'c}.
\newblock An efficient method for estimating the optimal dampers' viscosity for
  linear vibrating systems using {L}yapunov equation.
\newblock {\em {SIAM} J. Matrix Anal. Appl.}, 31(1):18--39, 2009.
\newblock \href {https://doi.org/10.1137/070683052}
  {\path{doi:10.1137/070683052}}.

\bibitem{WanYWetal25}
X.~Wang, X.~Yang, X.~Wang, and B.~Song.
\newblock Data-driven balanced truncation for second-order systems via the
  approximate {G}ramians.
\newblock e-print 2506.03855, arXiv, 2025.
\newblock Numerical Analysis (math.NA).
\newblock \href {https://doi.org/10.48550/arXiv.2506.03855}
  {\path{doi:10.48550/arXiv.2506.03855}}.

\bibitem{Wer21}
S.~W.~R. Werner.
\newblock {\em Structure-Preserving Model Reduction for Mechanical Systems}.
\newblock {D}issertation, Otto-von-Guericke-Universit{\"a}t, Magdeburg,
  Germany, 2021.
\newblock \href {https://doi.org/10.25673/38617} {\path{doi:10.25673/38617}}.

\bibitem{WerGG22}
S.~W.~R. Werner, I.~V. Gosea, and S.~Gugercin.
\newblock Structured vector fitting framework for mechanical systems.
\newblock {\em IFAC-Pap.}, 55(20):163--168, 2022.
\newblock 10th Vienna International Conference on Mathematical Modelling
  {MATHMOD} 2022.
\newblock \href {https://doi.org/10.1016/j.ifacol.2022.09.089}
  {\path{doi:10.1016/j.ifacol.2022.09.089}}.

\bibitem{Wya12}
S.~Wyatt.
\newblock {\em Issues in Interpolatory Model Reduction: {I}nexact Solves,
  Second-order Systems and {DAE}s}.
\newblock PhD thesis, Virginia Polytechnic Institute and State University,
  Blacksburg, Virginia, USA, 2012.
\newblock URL: \url{http://hdl.handle.net/10919/27668}.

\bibitem{YueM12}
Y.~Yue and K.~Meerbergen.
\newblock Using {K}rylov-{P}ad{\'e} model order reduction for accelerating
  design optimization of structures and vibrations in the frequency domain.
\newblock {\em Int. J. Numer. Methods Eng.}, 90(10):1207--1232, 2012.
\newblock \href {https://doi.org/10.1002/nme.3357}
  {\path{doi:10.1002/nme.3357}}.

\end{thebibliography}
  
\end{document}